\documentclass[11pt]{article}

\usepackage{geometry}
\usepackage{graphicx}
\usepackage{epstopdf}
\usepackage{booktabs} 
\usepackage{array} 
\usepackage{paralist} 
\usepackage{verbatim} 
\usepackage{subfigure} 
\usepackage{caption}
\usepackage{fancyhdr, bm} 
\usepackage[nottoc,notlof,notlot]{tocbibind}
\usepackage[titles,subfigure]{tocloft}

\usepackage{amsmath,amsfonts,amsthm,mathrsfs,amssymb,cite}
\usepackage[usenames]{color}
\usepackage{mathtools}
\usepackage{float}
\usepackage{hyperref}
\usepackage{dsfont}

\newtheorem{theorem}{Theorem}
\newtheorem{lemma}{Lemma}

\newtheorem*{ip}{Inverse Problem}
\theoremstyle{definition}
\newtheorem{defn}{Definition}[section]

\newtheorem{rem}{Remark}[section]
\newtheorem{example}{Example}
\numberwithin{equation}{section}
 \numberwithin{equation}{section}
  \numberwithin{theorem}{section}
   \numberwithin{lemma}{section}
    \numberwithin{defn}{section}
      \numberwithin{rem}{section}

\title{\bf  Fourier method for identifying  electromagnetic sources with multi-frequency far-field data}

\author{
 Xianchao Wang \thanks{Department of Mathematics, Harbin Institute of Technology, Harbin, China. Email: {\tt xcwang90@gmail.com}},
\and Minghui Song\thanks{Department of Mathematics, Harbin Institute of Technology, Harbin, China. Email: {\tt songmh@hit.edu.cn}},
\and Yukun  Guo \thanks{Department of Mathematics,
 Harbin Institute of Technology, Harbin, China. Email: {\tt ykguo@hit.edu.cn}},
\and Hongjie Li \thanks{Department of Mathematics, Hong Kong Baptist University, Kowloon, Hong Kong SAR, China.  Email:{\tt hongjie$_{-}$li@yeah.net}},
\and Hongyu Liu \thanks{Department of Mathematics, Hong Kong Baptist University, Kowloon, Hong Kong SAR, China.  Email:  {\tt hongyuliu@hkbu.edu.hk}}
\and}

\date{} 
\begin{document}
\maketitle

\begin{abstract}

  We consider the inverse problem of determining an unknown vectorial source current distribution associated with the homogeneous Maxwell system. We propose a novel non-iterative reconstruction method for solving the aforementioned inverse problem from far-field measurements. The method is based on recovering the Fourier coefficients of the unknown source. A key ingredient of the method is to establish the relationship between the Fourier coefficients and the multi-frequency far-field data. Uniqueness and stability results are established for the proposed reconstruction method. Numerical experiments are presented to illustrate the effectiveness and efficiency of the method.

\medskip

\medskip

\noindent{\bf Keywords:}~~inverse source problem, Maxwell's system, Fourier expansion, multi-frequency, far-field

\noindent{\bf 2010 Mathematics Subject Classification:}~~35R30, 35P25, 78A46

\end{abstract}

\section{Introduction}\label{sect:1}

The inverse source problem is concerned with the reconstruction of an unknown/inaccess-ible active source from the measurement of the radiating field induced by the source. The inverse source problem arises in many important applications including acoustic tomography \cite{Anastasio, ClaKli, Klibanov2013, Liu2015}, medical imaging\cite{Ammari2002, Arridge1999,Fokas2004}  and detection of pollution for the environment\cite{Badia2002}. In this paper, we are mainly concerned with the inverse source problem for wave propagation in the time-harmonic regime. In the last decades, many theoretical and numerical studies have been done in dealing with the inverse source problem for wave scattering.
The uniqueness and stability results can be found in \cite{Bao2017, Isa2 }. Several numerical reconstruction methods have also been proposed and developed in the literature. For a fixed frequency, we refer the reader to \cite{Ammari2002,Badia2013,He1998}.
However, with only one single frequency, the inverse source problem lacks of stability and it leads to severe ill-posedness.
In order to  improve the resolution,  multi-frequency measurements should be employed in the reconstruction \cite{Bao2017,Eller2009, Valdivia2012}.

The goal of this paper is to develop a novel numerical scheme for reconstructing an electric current source associated with the time-harmonic Maxwell system. Due to the existence of non-radiating sources \cite{Bleistein1977, Marengo2004}, the vectorial current sources cannot be uniquely determined from surface measurements. Albanese and Monk \cite{1Mon} showed that surface currents and dipole sources have a unique solution, but it is not valid for volume currents. Valdivia\cite{Valdivia2012} showed that
the volume currents could be uniquely identified if the current density is divergence free. Following the spirit of our earlier work \cite{WangGuo17,Wang17} by three of the authors of using Fourier method for inverse acoustic source problem, we develop a Fourier method for the reconstruction of a volume current associated with the time-harmonic Maxwell system. The extension from the scalar Hemholtz equation to the vectorial Maxwell system involves much subtle and technical analysis. First, we establish the one-to-one correspondence between the Fourier coefficients and  the far-field data, so that the Fourier coefficients can be directly calculated. Second, the proposed method is stable and robust to measurement noise. This is rigorously verified by establishing the corresponding stability estimates. Finally, compared to near-field Fourier method, our method is easy to implement with cheaper computational costs.

The rest of the paper is organized as follows. Section 2 describes the mathematical setup of the inverse source problem of our study. The theoretical uniqueness and stability results of proposed Fourier method are given in Section 3 and Section 4, respectively. Section 5 presents several numerical examples to illustrate the effectiveness and efficiency of the proposed method.

\section{Problem formulation}

Consider the following time-harmonic Maxwell system in $\mathbb{R}^3$,
\begin{equation}\label{eq: Maxwell}
\left\{
\begin{aligned}
&\nabla\times \bm E-\mathrm{i}\omega \mu_0 \bm H=0, \\
&\nabla\times\bm H+\mathrm{i}\omega\varepsilon_0\bm E=\bm J,
\end{aligned}
\right.
\end{equation}
with the Silver-M\"uller radiation condition
\begin{equation*}
  \lim_{|\bm x| \rightarrow +\infty} |\bm x|\left(\sqrt{\mu_0}\bm H\times \hat{\bm x}-\sqrt{\varepsilon_0}\bm E\right)=0,
\end{equation*}
where $\hat{\bm x}=\bm x/ |\bm x|$ and $\bm x=(x_1, x_2, x_3)^{\top}\in \mathbb{R}^3$.
Throughout the rest of the paper, we use non-bold and bold fonts to signify scalar and vectorial quantities, respectively. In \eqref{eq: Maxwell}, $\bm E$ denotes the electric filed, $ \bm H $ denotes the magnetic filed, $ \bm J $ is an electric current density, $\omega$ denotes the frequency, $\varepsilon_0$ denotes the electric permittivity and $\mu_0$ denotes the  magnetic permeability of the isotropic homogeneous background medium.
By eliminating $\bm H$ or $\bm E$ in \eqref{eq: Maxwell}, we obtain
\begin{equation*}
  \nabla\times\nabla\times \bm E-k^2 \bm E=\mathrm{i}\omega \mu_0 \bm J,
\end{equation*}
and
\begin{equation*}
  \nabla\times\nabla\times \bm H -k^2 \bm H= \nabla\times\bm J,
\end{equation*}
where $ k:=\omega\sqrt{\mu_0\varepsilon_0} $.
With the help of the vectorial Green function \cite{Tai94}, the radiated field can be written as
\begin{equation}\label{eq:a1}
  \displaystyle \bm E(\bm x)= \mathrm{i}\omega \mu_0  \left(\bm I+\frac{1}{k^2} \nabla \nabla\cdot\right)\int_{\mathbb{R}^3} \Phi(\bm x,\bm y)\, \bm J(\bm y) \,\mathrm{d}\bm y,
\end{equation}
and
\begin{equation}\label{eq:a2}
  \displaystyle \bm H(\bm x)= \nabla\times \int_{\mathbb{R}^3} \Phi(\bm x,\bm y)\, \bm J(\bm y) \,\mathrm{d}\bm y,
\end{equation}
respectively, where $\bm I$ is the $3\times 3$ identity matrix and
\begin{equation*}
  \displaystyle\Phi(\bm x,\bm y)=\frac{\mathrm{e}^{\mathrm{i}k|\bm x-\bm y|}}{4\pi|\bm x-\bm y|}, \quad \bm x\neq\bm y,
\end{equation*}
 is the fundamental solution to the Helmholtz equation.
The radiating fields $\bm E, \bm H $ to the Maxwell system have the following asymptotic expansion \cite{Colton2013}
\begin{equation*}
\begin{aligned}
  \bm E(\bm x)=\frac{e^{\mathrm{i}k|\bm x|}}{|\bm x|} \left\{ \bm E_{\infty}(\hat{\bm x})+ \mathcal{O}\left(\frac{1}{|\bm x|}\right)  \right\},\quad |\bm x|\rightarrow +\infty,\\
    \bm H(\bm x)=\frac{e^{\mathrm{i}k|\bm x|}}{|\bm x|} \left\{ \bm H_{\infty}(\hat{\bm x})+ \mathcal{O}\left(\frac{1}{|\bm x|}\right)  \right\},\quad |\bm x|\rightarrow +\infty,
    \end{aligned}
\end{equation*}
and by using the integral representations \eqref{eq:a1} and \eqref{eq:a2}, we have
\begin{align}
  &\label{eq:E_far}\bm E_{\infty}(\hat{\bm x})= \frac{\mathrm{i}\omega \mu_0}{4\pi}  \left(\bm I-\hat{\bm x} \hat{\bm x}^{\top}  \right)\int_{\mathbb{R}^3} e^{-\mathrm{i}k\hat{\bm x}\cdot \bm y} \, \bm J(\bm y) \,\mathrm{d}\bm y,\\
  &\label{eq:H_far}  \bm H_{\infty}(\hat{\bm x})=\frac{\mathrm{i}k}{4\pi} \hat{\bm x}\times\int_{\mathbb{R}^3} e^{-\mathrm{i}k\hat{\bm x}\cdot \bm y} \, \bm J(\bm y) \,\mathrm{d}\bm y.
\end{align}

In what follows, we always assume that  the electromagnetic source is a volume current that is supported in $D$.
As mentioned earlier, there exists non-radiating sources that produce no radiating field outside $D$. Hence, without any a prior knowledge, one can only recover the radiating part of the current density distribution. In order to formulate the uniqueness result, we assume that the current density distribution $\bm J$ only consists of radiating source, which is  independent of the wavenumber $k$ and of the form
\begin{equation*}
  \bm J\in \left(L^2(\mathbb{R}^3)\right)^3, \quad \mathrm{supp}\ \bm J\subset D,
\end{equation*}
where $D$ is a cube. Furthermore, the current density distribution $\bm J$ satisfies the transverse electric (TE) and  transverse magnetic (TM) decomposition; that is, the source can be expressed in the form
\begin{equation}\label{eq:density}
  \bm J=\bm p f+ \bm p\times\nabla g,
\end{equation}
where  $f\in L^2(D)$ and $ g\in H^1(D)$. We also refer to \cite{Lindell88} for more details on the TE/TM decomposition.
Here, $\bm p$ is the polarization direction which is assumed to be known and yields the following admissible set
\begin{equation}\label{eq:p}
  \mathbb{P}:=\left\{ \bm p \in \mathbb{S}^2 \mid \bm p\times \bm l \neq \bm 0, \quad  \forall \ \bm l\in \mathbb{Z}^3\backslash\{\bm 0\}    \right\}.
\end{equation}
From \eqref{eq:E_far} and \eqref{eq:H_far}, it is clear that
\begin{equation*}
\begin{aligned}
  &\bm E_{\infty}(-\hat{\bm x})=-\overline{\bm E_{\infty}(\hat{\bm x})},
  &\bm H_{\infty}(-\hat{\bm x})=\overline{\bm H_{\infty}(\hat{\bm x})},
  \end{aligned}
\end{equation*}
where and also in what follows, the overbar stands for the complex conjugate in this paper. Therefore, for our inverse problem, the measurements of the far-field data could be from an upper hemisphere $\mathbb{S}_{+}^2$,  say $x_3\geq0$. Figure \ref{fig:geometry} provides a schematic illustration of the geometric setting of the measurements. With the above discussion, the inverse source problem of the current study can be stated as follows,
\begin{ip}
 {\rm Given a fixed polarization direction $\bm p\in \mathbb{R}^3$ and a  finite number of wavenumbers $\{k\}$, we intend to recover the electromagnetic source $\bm J$ defined in \eqref{eq:density} from the electric far-field data  $\{\bm E_\infty(\hat{\bm x}_k;k,\bm p)\}$ or the magnetic far-field data $\{\bm H_\infty(\hat{\bm x}_k;k,\bm p)\}$, where $\hat{\bm x}_k$ depends on the wavenumber $k$ and $\hat{\bm x}_k\in \mathbb{S}_{+}^2$.}
\end{ip}

\begin{figure}
\centering
{\includegraphics[width=0.6\textwidth]
                   {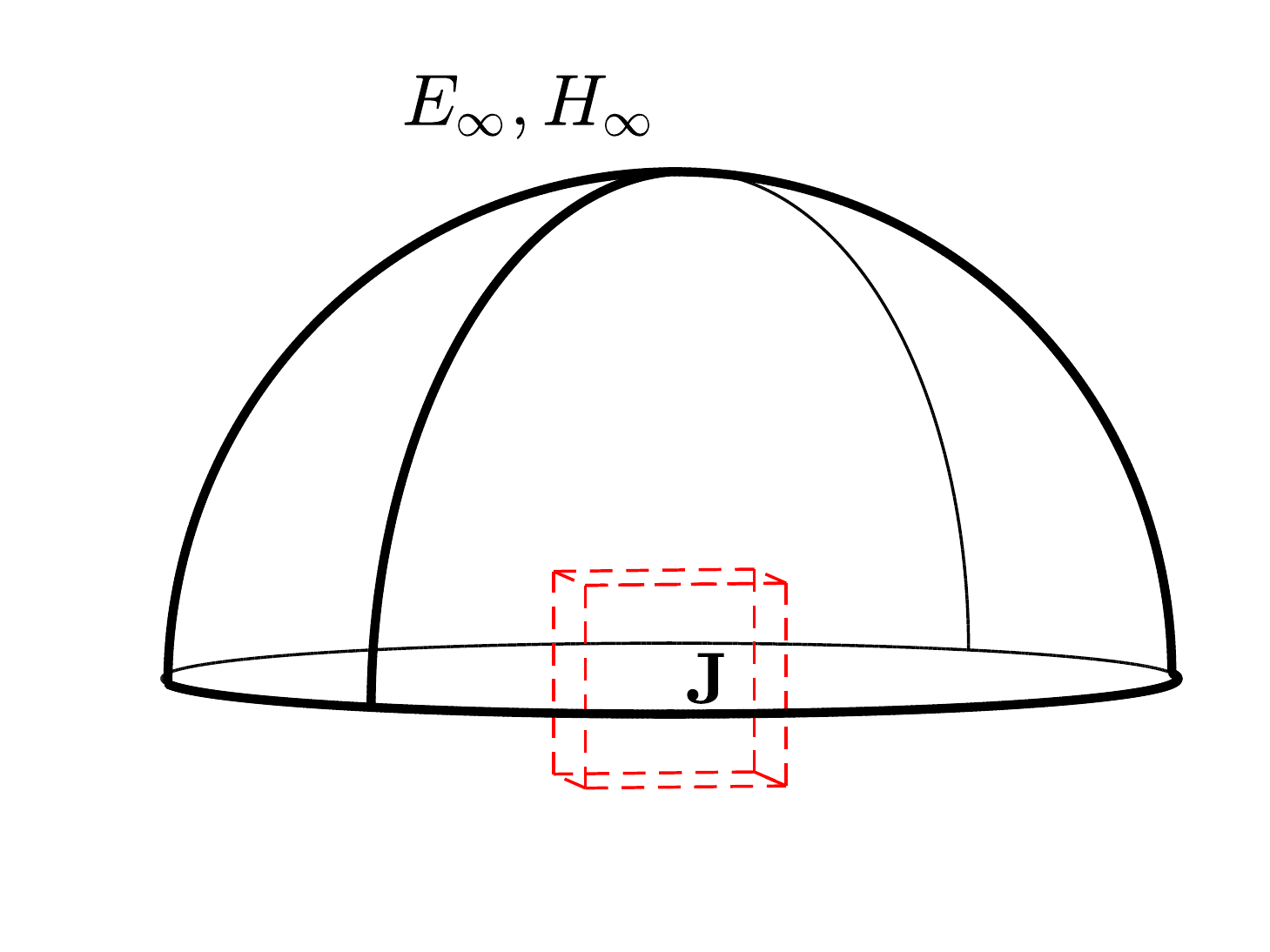}}
\caption{\label{fig:geometry} The schematic illustration of the inverse electromagnetic source problem by the far-field measurements with $x_3\geq0$.}
\end{figure}

\section{Uniqueness}

Prior to our discussion, we introduce some notations and relevant Sobolev spaces.
Without loss of generality, we let
\begin{equation*}
  D=\left(-\frac{a}{2},\ \frac{a}{2}\right)^3, \quad a \in \mathbb{R}_+.
\end{equation*}
Introduce the Fourier basis functions that are defined by
\begin{equation}\label{eq:Fourier_basic}
 \displaystyle \phi_{\bm l}(\bm x)=\exp\left(\mathrm{i}\frac{2\pi}{a} \bm l \cdot \bm x \right),\quad \bm l\in \mathbb{Z}^3, \ \bm x\in \mathbb{R}^3.
\end{equation}
By using the Fourier series expansion, the scalar functions $f\in L^2(D)$ and $g\in H^1(D)$ can be written as
\begin{equation*}
  f=\sum_{\bm l \in  \mathbb{Z}^3} \hat{f}_{\bm l}\, \phi_{\bm l}, \quad
  g=\sum_{\bm l \in  \mathbb{Z}^3\backslash \{\bm 0\}} \hat{g}_{\bm l}\, \phi_{\bm l},
\end{equation*}
where the Fourier coefficients are given by
\begin{align}
  &\label{eq:f_hat0}\hat{f}_{\bm l}=\frac{1}{a^3} \int_{D}f(\bm x)\overline{\phi_{\bm l}(\bm x)}\,\mathrm{d}\bm x,\\
  &\label{eq:g_hat0}\hat{g}_{\bm l}=\frac{1}{a^3} \int_{D}g(\bm x)\overline{\phi_{\bm l}(\bm x)}\,\mathrm{d}\bm x.
  \end{align}
Therefore the Fourier expansion of the current density $\bm J$ is
\begin{equation}\label{eq:J_expansion}
  \bm J=\bm p f+\bm p\times \nabla g=\bm p \sum_{\bm l \in  \mathbb{Z}^3} \hat{f}_{\bm l}\, \phi_{\bm l} + \frac{2\pi \mathrm{i}}{a} \sum_{\bm l \in  \mathbb{Z}^3\backslash \{\bm 0\}} (\bm p\times\bm l)\ \hat{g}_{\bm l}\, \phi_{\bm l}.
\end{equation}
The proposed reconstruction scheme in the current article is based on determining the Fourier coefficients $\hat{f}_{\bm l}$ and $\hat{g}_{\bm l}$ of the current density by using the corresponding electric or magnetic far-field data. For the subsequent use, we introduce the Sobolev spaces with $\sigma>0$
\begin{equation*}
  (H_{\bm p}^\sigma(D))^3:=\left\{ \bm p f+\bm p\times \nabla g \mid f \in H^\sigma(D), \, g\in H^{\sigma+1}(D), \, \bm p\in \mathbb{S}^2
    \right\},
\end{equation*}
equipped with the norm
\begin{equation*}
 \|\bm G\|_{\bm p , \sigma}=\left(\sum_{{\bm l}\in \mathbb{Z}^3}\left( 1+|{\bm l}|^2 \right)^{\sigma}|\hat{f}_{\bm l}|^2 + \frac{4\pi^2}{a^2}   \sum_{{\bm l}\in \mathbb{Z}^3\backslash \{\bm 0\}}\left( 1+|{\bm l}|^2 \right)^{\sigma}|\bm p\times \bm l|^2 |\hat{g}_{\bm l}|^2
\right)^{1/2}.
\end{equation*}
In addition, the wavenumber cannot be zero in \eqref{eq:E_far} and \eqref{eq:H_far}. Following \cite{WangGuo17}, we introduce the following definition of wavenumbers.

\begin{defn}[Admissible wavenumbers]\label{def:31}
Let  $\lambda$ be a sufficiently small positive constant and  the admissible wavenumbers can be defined by
\begin{equation}\label{eq:wavenumber}
  k_{\bm l}:=\left\{
\begin{aligned}
 & \frac{2\pi}{a}|\bm l|,    \quad \bm l \in \mathbb{Z}^3\backslash\{\bm 0\}, \\
 &  \frac{2\pi}{a}\lambda, \quad  \bm l =\bm 0.
\end{aligned}
\right.
\end{equation}
Correspondingly, the observation direction is given by
\begin{equation}\label{eq:x_hat}
  \hat{\bm x}_{\bm l}:=\left\{
\begin{aligned}
 & \hat{\bm l}, \quad   \quad \bm l \in \mathbb{Z}^3\backslash\{\bm 0\}, \\
 &  (1,0,0), \quad \bm l =\bm 0.
\end{aligned}
\right.
\end{equation}
\end{defn}
By virtue of Definition~\ref{def:31},  the Fourier basis functions defined in \eqref{eq:Fourier_basic} could be written as
\begin{equation*}
 \displaystyle \phi_{\bm l}(\bm x)= \exp\left(\mathrm{i}k_{\bm l}\, \hat{\bm l} \cdot \bm x \right),\quad \bm l\in \mathbb{Z}^3, \ \bm x\in \mathbb{R}^3.
\end{equation*}

Next we state the uniqueness result.

\begin{theorem}
Let $ k_{\bm l}$ and $ \hat{x}_{\bm l}$  be defined in \eqref{eq:wavenumber} and \eqref{eq:x_hat}, then the Fourier coefficients $\{\hat{f}_{\bm l}\}$  and $\{\hat{g}_{\bm l}\}$ in \eqref{eq:f_hat0} and \eqref{eq:g_hat0} could be uniquely  determined by $ \{\bm E_{\infty}(\hat{\bm x}_{\bm l};k_{\bm l},\bm p)\}$ or $ \{\bm H_{\infty}(\hat{\bm x}_{\bm l};k_{\bm l},\bm p)\}$,  where  $  \bm l \in \mathbb{Z}^3  $.
\end{theorem}
\begin{proof}[\bf Proof.]
Let $\bm J$ be the electromagnetic source that produces the electric far-field  data $\{\bm E_{\infty}(\hat{\bm x}_{\bm l};k_{\bm l})\}_{\bm l \in \mathbb{Z}^3}$  and the magnetic far-field  data $\{\bm H_{\infty}(\hat{\bm x}_{\bm l};k_{\bm l})\}_{\bm l \in \mathbb{Z}^3}$ on $\mathbb{S}^2$.

First, we consider the recovery of $\bm J$ by the magnetic far-field  data.
For every $\bm l\in \mathbb{Z}^3\backslash\{\bm 0\}$, using \eqref{eq:H_far} and \eqref{eq:J_expansion},  we have
 \begin{equation}\label{eq:H_far1}
 \begin{aligned}
   &\bm H_{\infty}(\hat{\bm x}_{\bm l}; k_{\bm l})\\
   &=\frac{\mathrm{i}k_{\bm l}}{4\pi} \hat{\bm x}_{\bm l}
   \times \int_D \left(\bm p \hat{f}_{\bm 0} \mathrm{e}^{-\mathrm{i}k_{\bm l}\hat{\bm x}_{\bm l}\cdot \bm y}
   + \sum_{\tilde{\bm l} \in  \mathbb{Z}^3\backslash \{\bm 0\}}\left(\bm p \hat{f}_{\tilde{\bm l}}+ \frac{2\pi \mathrm{i}}{a} (\bm p\times\tilde{\bm l})\hat{g}_{\tilde{\bm l} }\right)\mathrm{e}^{\mathrm{i} (k_{\tilde{\bm l}}\hat{\tilde{\bm l}}-k_{\bm l}\hat{\bm x}_{\bm l})\cdot \bm y}\right) \mathrm{d}\bm y\\
   &=\frac{\mathrm{i}k_{\bm l} a^3}{4\pi} \left( \hat{\bm x}_{\bm l}
   \times \bm p \hat{f}_{\bm l} + \frac{2\pi \mathrm{i}}{a}  \hat{\bm x}_{\bm l}
   \times (\bm p\times \bm l)  \hat{g}_{\bm l} \right).
   \end{aligned}
 \end{equation}
 From \eqref{eq:p} and \eqref{eq:x_hat}, we see that $\{ \hat{\bm x}_{\bm l}, \bm p\times \hat{\bm x}_{\bm l}, \hat{\bm x}_{\bm l}\times (\bm p\times \hat{\bm x}_{\bm l}) \}$
forms an orthogonal basis of $\mathbb{R}^{3}$. Multiplying $ \hat{\bm x}_{\bm l}  \times \bm p$ on the both sides of \eqref{eq:H_far1}, and using the orthogonality, we obtain
   \begin{equation}\label{eq:Hf_hat}
     \hat{f}_{\bm l}=\frac{4\pi \hat{\bm x}_{\bm l}\times \bm p \cdot\bm H_{\infty}(\hat{\bm x}_{\bm l}; k_{\bm l})}{\mathrm{i}k_{\bm l} a^3
     |\hat{\bm x}_{\bm l} \times \bm p|^2}.
   \end{equation}
  Similarly, multiplying $ \hat{\bm x}_{\bm l}\times (\bm p\times \bm l)$ on the both sides of \eqref{eq:H_far1}, we have
  \begin{equation}\label{eq:Hg_hat}
     \hat{g}_{\bm l}=-\frac{2 \hat{\bm x}_{\bm l}\times (\bm p\times \bm l) \cdot\bm H_{\infty}(\hat{\bm x}_{\bm l}; k_{\bm l})}{k_{\bm l} a^2
     |\hat{\bm x}_{\bm l} \times (\bm p\times\bm l)|^2}.
   \end{equation}
 For  $\bm l=\bm 0$,  we have
  \begin{equation*}
 \begin{aligned}
   &\bm H_{\infty}(\hat{\bm x}_{\bm 0}; k_{\bm 0})\\
   &=\frac{\mathrm{i}k_{\bm 0}}{4\pi} \hat{\bm x}_{\bm 0}
   \times \int_D \left( \bm p \hat{f}_{\bm 0} \mathrm{e}^{-\mathrm{i}k_{\bm 0}\hat{\bm x}_{\bm 0}\cdot \bm y}+ \sum_{\bm l \in  \mathbb{Z}^3\backslash \{\bm 0\}}\left(\bm p \hat{f}_{\bm l}+ \frac{2\pi \mathrm{i}}{a} (\bm p\times\bm l)\hat{g}_{\bm l}\right)\mathrm{e}^{\mathrm{i}(k_{\bm l}\hat{\bm l}-k_{\bm 0}\hat{\bm x}_{\bm 0})\cdot \bm y} \right)\mathrm{d}\bm y.
   \end{aligned}
 \end{equation*}
 Multiplying $ \hat{\bm x}_{\bm 0}  \times \bm p$ on the both side of the last equation, and also using the orthogonal property, we obtain
 \begin{equation*}
 \begin{aligned}
   &\hat{\bm x}_{\bm 0}  \times \bm p \cdot \bm H_{\infty}(\hat{\bm x}_{\bm 0}; k_{\bm 0})\\
   &=\frac{\mathrm{i}k_{\bm 0}}{4\pi} |\hat{\bm x}_{\bm 0}
   \times \bm p|^2 \left( a^3\hat{f}_{\bm 0} \frac{\sin \lambda \pi}{\lambda \pi}+ \sum_{\bm l \in  \mathbb{Z}^3\backslash \{\bm 0\}}\hat{f}_{\bm l} \int_{D}
  \mathrm{e}^{\mathrm{i}(k_{\bm l}\hat{\bm l}-k_{\bm 0}\hat{\bm x}_{\bm 0})\cdot \bm y} \, \mathrm{d}\bm y  \right).
   \end{aligned}
 \end{equation*}
 Thus,
   \begin{equation}\label{eq:Hf0_hat}
     \hat{f}_{\bm 0}= \frac{\lambda \pi}{a^3 \sin \lambda \pi} \left( \frac{4\pi \hat{\bm x}_{\bm 0}\times \bm p \cdot\bm H_{\infty}(\hat{\bm x}_{\bm 0}; k_{\bm 0})}{\mathrm{i}k_{\bm 0}
     |\hat{\bm x}_{\bm 0} \times \bm p|^2}- \sum_{\bm l \in  \mathbb{Z}^3\backslash \{\bm 0\}}\hat{f}_{\bm l} \int_{D}
  \mathrm{e}^{\mathrm{i}(k_{\bm l}\hat{\bm l}-k_{\bm 0}\hat{\bm x}_{\bm 0})\cdot \bm y} \, \mathrm{d}\bm y  \right).
   \end{equation}

Next, we consider the recovery of $\bm J$ by the electric far-field  data.
For every $\bm l\in \mathbb{Z}^3\backslash\{\bm 0\}$, using \eqref{eq:E_far} and \eqref{eq:J_expansion},  we have
 \begin{equation}\label{eq:E_far1}
 \begin{aligned}
   \bm E_{\infty}(\hat{\bm x}_{\bm l}; k_{\bm l})
   =\frac{\mathrm{i}\omega \mu_0 a^3}{4\pi}\left( \bm I -\hat{\bm x}_{\bm l} \hat{\bm x}_{\bm l} ^{\top} \right) \left(  \bm p \hat{f}_{\bm l} + \frac{2\pi \mathrm{i}}{a}  (\bm p\times \bm l)  \hat{g}_{\bm l} \right).
   \end{aligned}
 \end{equation}
Through straightforward calculations, one can verify that
 \begin{equation*}
  \hat{\bm x}_{\bm l} \times(\hat{\bm x}_{\bm l} \times \bm A)=-\left( \bm I -\hat{\bm x}_{\bm l}  \hat{\bm x}_{\bm l} ^{\top} \right)\bm A,  \quad \forall\bm A\in \mathbb{R}^3.
 \end{equation*}
Combining the last two equations, one can show that
  \begin{equation}\label{eq:E_far2}
 \begin{aligned}
   &\bm E_{\infty}(\hat{\bm x}_{\bm l}; k_{\bm l})
   =\frac{\mathrm{i}\omega \mu_0 a^3}{4\pi} \left(  -\hat{\bm x}_{\bm l} \times(\hat{\bm x}_{\bm l} \times\bm p) \hat{f}_{\bm l} + \frac{2\pi \mathrm{i}}{a}  ( -\hat{\bm x}_{\bm l} \times(\hat{\bm x}_{\bm l} \times(\bm p \times \bm l))  \hat{g}_{\bm l} \right).
   \end{aligned}
 \end{equation}
 Multiplying $ \bm p $ on the both sides of \eqref{eq:E_far2}, and using the orthogonality, we obtain
 \begin{equation*}
 \begin{aligned}
   &\bm p\cdot \bm E_{\infty}(\hat{\bm x}_{\bm l}; k_{\bm l})\\
   &=\frac{\mathrm{i}\omega \mu_0 a^3}{4\pi}\left(  -\bm p\cdot\hat{\bm x}_{\bm l} \times(\hat{\bm x}_{\bm l} \times\bm p) \hat{f}_{\bm l} + \frac{2\pi \mathrm{i}}{a}  ( -\bm p\cdot\hat{\bm x}_{\bm l} \times(\hat{\bm x}_{\bm l} \times(\bm p \times \bm l))  \hat{g}_{\bm l} \right)\\
   &=\frac{\mathrm{i}\omega \mu_0 a^3}{4\pi}\left(  (\hat{\bm x}_{\bm l}\times \bm p) \cdot(\hat{\bm x}_{\bm l} \times\bm p) \hat{f}_{\bm l} + \frac{2\pi \mathrm{i}}{a}   (\hat{\bm x}_{\bm l}\times \bm p) \cdot(\hat{\bm x}_{\bm l} \times(\bm p \times \bm l))  \hat{g}_{\bm l} \right)\\
   &=\frac{\mathrm{i}\omega \mu_0 a^3}{4\pi}\left|  \hat{\bm x}_{\bm l}\times \bm p \right|^2 \hat{f}_{\bm l}.
   \end{aligned}
 \end{equation*}
 Thus,
   \begin{equation}\label{eq:Ef_hat}
     \hat{f}_{\bm l}=\frac{4\pi  \bm p \cdot\bm E_{\infty}(\hat{\bm x}_{\bm l}; k_{\bm l})}{\mathrm{i}\omega \mu_0 a^3
     |\hat{\bm x}_{\bm l} \times \bm p|^2}.
   \end{equation}
  Similarly, multiplying $ \bm p\times \bm l$ on the both sides of \eqref{eq:E_far2}, we obtain
  \begin{equation}\label{eq:Eg_hat}
     \hat{g}_{\bm l}=-\frac{2 (\bm p\times \bm l) \cdot\bm E_{\infty}(\hat{\bm x}_{\bm l}; k_{\bm l})}{\omega\mu_0 a^2
     |\hat{\bm x}_{\bm l} \times (\bm p\times\bm l)|^2}.
   \end{equation}
 For  $\bm l=\bm 0$,  we have
  \begin{equation*}
 \begin{aligned}
   &\bm E_{\infty}(\hat{\bm x}_{\bm 0}; k_{\bm 0})\\
   &=\frac{\mathrm{i}\omega \mu_0}{4\pi} \left( \bm I -\hat{\bm x}_{\bm 0}  \hat{\bm x}_{\bm 0} ^{\top} \right) \int_D \left( \bm p \hat{f}_{\bm 0} \mathrm{e}^{-\mathrm{i}k_{\bm 0}\hat{\bm x}_{\bm 0}\cdot \bm y}+ \sum_{\bm l \in  \mathbb{Z}^3\backslash \{\bm 0\}}\left(\bm p \hat{f}_{\bm l}+ \frac{2\pi \mathrm{i}}{a} (\bm p\times\bm l)\hat{g}_{\bm l}\right)\mathrm{e}^{\mathrm{i}(k_{\bm l}\hat{\bm l}-k_{\bm 0}\hat{\bm x}_{\bm 0})\cdot \bm y} \right)\mathrm{d}\bm y.
   \end{aligned}
 \end{equation*}
 Multiplying $ \bm p$ on the both sides of the last equation, and also using the orthogonality, we obtain
 \begin{equation*}
 \begin{aligned}
   & \bm p \cdot \bm E_{\infty}(\hat{\bm x}_{\bm 0}; k_{\bm 0})
   =\frac{\mathrm{i}\omega \mu_0}{4\pi} |\hat{\bm x}_{\bm 0}
   \times \bm p|^2 \left( a^3\hat{f}_{\bm 0} \frac{\sin \lambda \pi}{\lambda \pi}+ \sum_{\bm l \in  \mathbb{Z}^3\backslash \{\bm 0\}}\hat{f}_{\bm l} \int_{D}
  \mathrm{e}^{\mathrm{i}(k_{\bm l}\hat{\bm l}-k_{\bm 0}\hat{\bm x}_{\bm 0})\cdot \bm y} \, \mathrm{d}\bm y  \right).
   \end{aligned}
 \end{equation*}
 Thus,
   \begin{equation*}
     \hat{f}_{\bm 0}= \frac{\lambda \pi}{a^3 \sin \lambda \pi} \left( \frac{4\pi  \bm p \cdot\bm E_{\infty}(\hat{\bm x}_{\bm 0}; k_{\bm 0})}{\mathrm{i} \omega \mu_0
     |\hat{\bm x}_{\bm 0} \times \bm p|^2}- \sum_{\bm l \in  \mathbb{Z}^3\backslash \{\bm 0\}}\hat{f}_{\bm l} \int_{D}
  \mathrm{e}^{\mathrm{i}(k_{\bm l}\hat{\bm l}-k_{\bm 0}\hat{\bm x}_{\bm 0})\cdot \bm y} \, \mathrm{d}\bm y  \right).
   \end{equation*}

The proof is complete.
\end{proof}

In practical computations,  we have to truncate the infinite series by a finite order $N \in \mathbb{N}$ to approximate  $\bm J$ by
\begin{equation}\label{eq:J_N}
  \bm J_N=\bm p \hat{f}_{\bm 0}+ \sum_{1\leq|\bm l|_{\infty}\leq N} \left(\bm p \hat{f}_{\bm l} + \frac{2\pi \mathrm{i}}{a}(\bm p\times\bm l) \,\hat{g}_{\bm l}\right) \phi_{\bm l},
\end{equation}
where $\hat{f}_{\bm 0}$ could be represented by magnetic far-field
\begin{equation}\label{eq:Hf0_hat}
     \hat{f}_{\bm 0}\approx \frac{\lambda \pi}{a^3\sin \lambda \pi} \left( \frac{4\pi \hat{\bm x}_{\bm 0}\times \bm p \cdot\bm H_{\infty}(\hat{\bm x}_{\bm 0}; k_{\bm 0})}{\mathrm{i}k_{\bm 0}
     |\hat{\bm x}_{\bm 0} \times \bm p|^2}- \sum_{1\leq|\bm l|_{\infty}\leq N} \hat{f}_{\bm l} \int_{D}
  \mathrm{e}^{\mathrm{i}(k_{\bm l}\hat{\bm l}-k_{\bm 0}\hat{\bm x}_{\bm 0})\cdot \bm y} \, \mathrm{d}\bm y  \right),
\end{equation}
or electric far-field
   \begin{equation}\label{eq:Ef0_hat}
     \hat{f}_{\bm 0}\approx \frac{\lambda \pi}{a^3 \sin \lambda \pi} \left( \frac{4\pi  \bm p \cdot\bm E_{\infty}(\hat{\bm x}_{\bm 0}; k_{\bm 0})}{\mathrm{i} \omega \mu_0
     |\hat{\bm x}_{\bm 0} \times \bm p|^2}- \sum_{1\leq|\bm l|_{\infty}\leq N}\hat{f}_{\bm l} \int_{D}
  \mathrm{e}^{\mathrm{i}(k_{\bm l}\hat{\bm l}-k_{\bm 0}\hat{\bm x}_{\bm 0})\cdot \bm y} \, \mathrm{d}\bm y  \right).
   \end{equation}

\section{Stability}
In this section, we derive the stability estimates of recovering the Fourier coefficients of the electric current source by using the far-field data. We only consider the stability of using the magnetic far-field data, and the case with the electric far-field data can be treated in a similar manner. In what follows, we introduce $H_{\infty}^{\delta}(\hat{\bm x}_{\bm l}; k_{\bm l})$ such that
\begin{equation*}
  |\bm H_{\infty}^{\delta}(\hat{\bm x}_{\bm l}; k_{\bm l})-\bm H_{\infty}(\hat{\bm x}_{\bm l}; k_{\bm l})|\leq \delta |\bm H_{\infty}(\hat{\bm x}_{\bm l}; k_{\bm l})|,
\end{equation*}
where $\delta>0$. We first present two auxiliary results.

\begin{theorem}
 For $ \bm{l} \in \mathbb{Z}^{3},\, |\bm{l}|_{\infty} \leq N$, we have
\begin{align}
     \label{eq:f_cefficient}  &  |\hat{f}^{\delta}_{\bm{l}}- \hat{f}_{\bm{l}}| \leq C_1 \delta, \quad \quad  1\leq |\bm{l}|_{\infty} \leq N, \\
     \label{eq:g_cefficient}  &  |\hat{g}^{\delta}_{\bm{l}}- \hat{g}_{\bm{l}}| \leq C_2 \delta, \quad \quad  1\leq |\bm{l}|_{\infty} \leq N, \\
     \label{eq:f_cefficient0}  &   |\hat{f}^{\delta}_{\bm 0}- \hat{f}_{\bm 0}|  \leq  C_3\delta +  C_4 \lambda N \delta+ C_5\frac{\lambda}{\sqrt{N}},
\end{align}
where constants $C_1, C_2, C_3, C_4$  and $C_5$ depend on $ f, g , a $ and $\lambda$.
\end {theorem}
\begin{proof}[\bf Proof.]
For $ \bm {l}\in \mathbb{Z}^3\backslash\{\bm 0\} $, from Schwarz inequality and \eqref{eq:Hf_hat}, we have
 \begin{align*}
   |\hat{f}^{ \delta}_{\bm {l}}- \hat{f}_{\bm {l}} |
   &= \left|\frac{4\pi \hat{\bm x }_{\bm l}\times \bm p }{\mathrm{i}k_{\bm l} a^3
   |\hat{\bm x}_{\bm l}\times \bm p |^2}\cdot \left(\bm H_{\infty}^{\delta}(\hat{\bm x}_{\bm l}; k_{\bm l})-\bm H_{\infty}(\hat{\bm x}_{\bm l}; k_{\bm l})\right)\right|\\
   &\leq \frac{4\pi }{\mathrm{i}k_{\bm l} a^3
   |\hat{\bm x}_{\bm l}\times \bm p |}\delta\left|\bm H_{\infty}(\hat{\bm x}_{\bm l}; k_{\bm l})\right|\\
   &\leq \frac{\delta}{ a^3
   |\hat{\bm x}_{\bm l}\times \bm p |}\left| \hat{\bm x}\times\int_{D} e^{-\mathrm{i}k\hat{\bm x}_{\bm l}\cdot \bm y} \, \bm J(y) \,\mathrm{d}\bm y  \right|\\
   &\leq \frac{\delta}{ a^3
   |\hat{\bm x}_{\bm l}\times \bm p |} |\hat{\bm x}\times \bm p| \left|\int_{D} e^{-\mathrm{i}k\hat{\bm x}_{\bm l}\cdot \bm y} \ ( f(\bm y)+|\nabla g(\bm y)|) \,\mathrm{d}\bm y  \right|\\
      &\leq \frac{\delta}{ a^3}  \left(\int_{D} \left|e^{-\mathrm{i}k\hat{\bm x}_{\bm l}\cdot \bm y}\right|^2 \,\mathrm{d}\bm y\right)^{1/2}   \left( \|f\|_{L^2(D)}+\|\nabla g\|_{L^2(D)}\right)  \\
      &\leq C_1 \delta
\end{align*}
where $ C_1 =(\|f\|_{L^2(D)}+\| g\|_{H^1(D)} )/{ a^{3/2}}$  and
it leads to estimate \eqref{eq:f_cefficient}.

Correspondingly, from \eqref{eq:Hg_hat}, we have
   \begin{align*}
  |\hat{g}^{ \delta}_{\bm {l}}- \hat{g}_{\bm {l}} |
   &=\left|-\frac{2 \hat{\bm x}_{\bm l}\times (\bm p\times \bm l) }{k_{\bm l} a^2 |\hat{\bm x}_{\bm l} \times (\bm p\times\bm l)|^2} \cdot \left(\bm H_{\infty}^{\delta}(\hat{\bm x}_{\bm l}; k_{\bm l})-\bm H_{\infty}(\hat{\bm x}_{\bm l}; k_{\bm l})\right)\right| \\
   &\leq \frac{2 }{k_{\bm l} a^2
   |\hat{\bm x}_{\bm l}\times (\bm p\times\bm l) |}\delta \left|\bm H_{\infty}(\hat{\bm x}_{\bm l}; k_{\bm l})\right|\\
   &\leq \frac{\delta}{ 2\pi |\bm l|a^2
   |\hat{\bm x}_{\bm l}\times \bm p |}\left| \hat{\bm x}\times\int_{D} e^{-\mathrm{i}k\hat{\bm x}_{\bm l}\cdot \bm y} \, \bm J(y) \,\mathrm{d}\bm y  \right|\\
   &\leq \frac{ \|f\|_{L^2(D)}+\| g\|_{H^1(D)} }{2\pi |\bm l| a^{1/2}}\  \delta\\
   &\leq C_2  \delta,
   \end{align*}
   where $C_2= (\|f\|_{L^2(D)}+\| g\|_{H^1(D)})/(2\pi a^{1/2})$ and it  verifies \eqref{eq:g_cefficient}.

For $ \bm {l}= \{\bm 0\} $, from Schwarz inequality and \eqref{eq:Hf0_hat}, we have
   \begin{align*}
     |\hat{f}_{\bm 0}^{\delta}-\hat{f}_{\bm 0}|
   \leq & \frac{\lambda \pi}{a^3 \sin \lambda \pi} \left|\frac{4\pi \hat{\bm x }_{\bm 0}\times \bm p }{\mathrm{i}k_{\bm 0}
   |\hat{\bm x}_{\bm 0}\times \bm p |^2}\cdot \left(\bm H_{\infty}^{\delta}(\hat{\bm x}_{\bm 0}; k_{\bm 0})-\bm H_{\infty}(\hat{\bm x}_{\bm 0}; k_{\bm 0})\right)\right|\\
   & +\underbrace{\frac{\lambda \pi}{ a^3 \sin \lambda \pi}\sum_{1\leq|\bm {l}|_{\infty}\leq N}\left|\left(\hat{f}^{\delta}_{\bm {l}}- \hat{f}_{\bm {l}}\right)  \int_{D}
  \mathrm{e}^{\mathrm{i}(k_{\bm l}\hat{\bm l}-k_{\bm 0}\hat{\bm x}_{\bm 0})\cdot \bm y} \, \mathrm{d}\bm y  \right|}_{I_1}\\
   &+\underbrace{\frac{\lambda \pi}{a^3 \sin \lambda \pi} \sum_{|\bm {l}|_{\infty}\geq N}\left|\hat{f}_{\bm {l}} \int_{D}
  \mathrm{e}^{\mathrm{i}(k_{\bm l}\hat{\bm l}-k_{\bm 0}\hat{\bm x}_{\bm 0})\cdot \bm y} \, \mathrm{d}\bm y \right|}_{I_2}\\
   \triangleq  & \ C_3  \delta+ I_{1}+I_{2}.
   \end{align*}
   where $ \displaystyle C_3={ \lambda \pi(\|f\|_{L^2(D)}+\| g\|_{H^1(D)})}/{(a^{9/2} \sin \lambda \pi)}$.

Define $\bm l=(l_1, l_2, l_3)\in\mathbb{ Z}^3$, from \eqref{eq:wavenumber} and \eqref{eq:x_hat},  we find that
\begin{equation*}
\int_{D} \mathrm{e}^{\mathrm{i}(k_{\bm l}\hat{\bm l}-k_{\bm 0}\hat{\bm x}_{\bm 0})\cdot \bm y} \, \mathrm{d}\bm y =\left\{
\begin{aligned}
 & \frac{a^3 \sin \, (l_1-\lambda)\pi}{(l_1-\lambda)\pi},    &|\bm l|=|l_1|, \\
 &  0, &|\bm l|\neq |l_1|,
\end{aligned}
\right.
\end{equation*}
which together with \eqref{eq:f_cefficient} gives
 \begin{align*}
 I_1 \leq &\frac{\lambda \pi}{a^3 \sin \lambda \pi}\sum_{1\leq|\bm {l}|_{\infty}\leq N}\left|\hat{f}^{\delta}_{\bm {l}}- \hat{f}_{\bm {l}}\right|  \left|\frac{a^3\sin (l_1-\lambda)\pi}{(l_1-\lambda)\pi} \right| \\
 \leq & \frac{\lambda \pi}{\sin \lambda \pi} 2 \sum_{j=1}^{N}\left( C_1 \delta \frac{ \sin \lambda\pi}{(j-\lambda)\pi}\right)\\
 \leq & C_4 \lambda N \delta,
 \end{align*}
where $C_4=2 C_1$.
On the other hand, one can deduce that
  \begin{align*}
     I_2 \leq &\frac{\lambda \pi}{a^3 \sin \lambda \pi}\sum_{|\bm {l}|_{\infty} > N}| \hat{f}_{\bm {l}}| \left|\frac{a^3 \sin (l_1-\lambda)\pi}{(l_1-\lambda)\pi}  \right| \\
     \leq & \frac{\lambda \pi}{\mathrm{sin}\,\lambda \pi}\left(\sum_{|\bm {l}|_{\infty} > N} | \hat{f}_{\bm {l}}|^2\right)^{1/2} \left(  \sum_{|\bm {l}|_{\infty} > N}\left|\frac{\sin (l_1-\lambda)\pi}{(l_1-\lambda)\pi}  \right|^2\right)^{1/2} \\
     \leq & \frac{\lambda \pi}{\mathrm{sin}\,\lambda \pi}\frac{1}{a^3}\left \|f \right\|_{L^2(D)}\left( 2\sum_{j=N+1}^\infty \left|\frac{ \mathrm{sin}\,\lambda\pi}{(j-\lambda)\pi}\right |^2 \right)^{1/2}\\
     \leq &\frac{2\lambda}{a^3\sqrt{N}}\|f\|_{L^2(D)} \\
     = & C_5 \frac{\lambda}{\sqrt{N}},
 \end{align*}
where $ C_5=2\|f\|_{L^2(D)}/a^3 $.  Finally, we obtain
\begin{equation*}
  |\hat{f}^{\delta}_{\bm 0}- \hat{f}_{\bm 0}|
  \leq   C_3\delta + C_4 \lambda N \delta+ C_5\frac{\lambda}{\sqrt{N}}.
\end{equation*}

The proof is complete.
\end{proof}

\begin{lemma}\label{a}\cite{Wang17}
    Let $\bm J$ be a vector function in $(H_{\bm p}^{\sigma}(D))^3 $ and $0\leq\mu\leq\sigma$, then the following estimate holds
     \begin{equation*}
      \|\bm J_N-\bm J\|_{\bm p, \mu}\leq N^{\mu-\sigma}\|\bm J \|_{\bm p,\sigma}, \quad  0\leq\mu\leq\sigma.
       \end{equation*}
        \end{lemma}

The stability result is contained in the following theorem.

\begin{theorem}\label{A-2.3}
Let $\bm J \in (H_{\bm p}^{\sigma}(D))^3 $ and $0\leq\mu\leq\sigma$, then the following estimate holds
\begin{equation*}
  \| \bm J_{N}^{\delta}- \bm J \|_{\bm p, \mu}
  \leq  C_6\delta +  C_6 \lambda N \delta+ C_6\frac{\lambda}{\sqrt{N}}+ C_7 N^{\mu+3/2}\delta+ C_8 N^{\mu+5/2}\delta+N^{\mu-\sigma}\|\bm J \|_{\bm p,\sigma},
\end{equation*}
where $C_6, C_7, C_8$ depend only on $f, g, a $ and $\lambda$.
\end{theorem}
\begin{proof}[\bf Proof.]
It is readily seen that
\begin{equation}\label{3.4}
  \begin{aligned}
  &\| \bm J_{N}^{ \delta}- \bm J_{N}\|_{\bm p, \mu}\\
  \leq &\left(\sum_{|\bm l|_{\infty}=0}^N \left( 1+|{\bm l}|^2 \right)^{\mu}|\hat{f}_{\bm l}^{\delta}-\hat{f}_{\bm l}|^2
  + \frac{4\pi^2}{a^2}   \sum_{|\bm l|_{\infty}=1}^N\left( 1+|{\bm l}|^2 \right)^{\mu}|\bm p\times \bm l|^2 |\hat{g}_{\bm l}^{\delta}-\hat{g}_{\bm l}|^2
\right)^{1/2}\\
  \leq & |\hat{f}_{\bm 0}^{\delta}-\hat{f}_{\bm 0}| + \left(\sum_{|\bm l|_{\infty}=1}^N \left( 1+|{\bm l}|^2 \right)^{\mu}|\hat{f}_{\bm l}^{\delta}-\hat{f}_{\bm l}|^2\right)^{1/2}\\
       &+ \left(\frac{4\pi^2 }{a^2}  \sum_{|\bm l|_{\infty}=1}^N\left( 1+|{\bm l}|^2 \right)^{\mu}|\bm p\times \bm l|^2 |\hat{g}_{\bm l}^{\delta}-\hat{g}_{\bm l}|^2 \right)^{1/2}\\
\leq & \left(C_3\delta +  C_4 \lambda N \delta + C_5\frac{\lambda}{\sqrt{N}} \right) +  C_1\delta \left(\sum_{|\bm l|_{\infty}=1}^N \left( 1+|{\bm l}|^2 \right)^{\mu}\right)^{1/2}\\
  &+\frac{2\pi}{a} C_2\delta\left(\sum_{|\bm {l}|_{\infty}=1 }^N \left( 1+|\bm {l}|^2\right)^{\mu}\left|\bm l \right|^2 \right)^{1/2}\\
  \leq &   C_6\delta +  C_6 \lambda N \delta+ C_6\frac{\lambda}{\sqrt{N}}+ C_7 N^{\mu+3/2}\delta+ C_8 N^{\mu+5/2}\delta,
   \end{aligned}
 \end{equation}
 where $C_6=\max\, \{C_3, C_4, C_5\}$.
Hence, from \eqref{3.4} and Lemma \ref{a} , we obtain
\begin{equation*}
  \| \bm J_{N}^{ \delta}- \bm J \|_{\bm p, \mu}
  \leq  C_6\delta +  C_6 \lambda N \delta+ C_6\frac{\lambda}{\sqrt{N}}+ C_7 N^{\mu+3/2}\delta+ C_8 N^{\mu+5/2}\delta+N^{\mu-\sigma}\|\bm J \|_{\bm p,\sigma},
\end{equation*}
which completes the proof.
\end{proof}

\begin{rem}\label{rem:N}
If one takes $ N=\tau \delta^{-\frac{1}{\sigma+5/2}}$ with $\tau\geq 1$ in Theorem \ref{A-2.3} , we have
\begin{equation*}
  \begin{aligned}
  \| \bm J_{N}^{\delta}-\bm J \|_{\bm p, \mu}
  \leq &  C_6\delta +C_6\lambda\tau\delta^{\frac{\sigma+3/2}{\sigma+5/2}} +\frac{C_6\lambda}{ \sqrt{\tau}}\delta^{\frac{1}{2\sigma+5}}
+ C_7 \tau^{\mu+3/2}\delta^{\frac{1+\sigma-\mu}{\sigma+5/2}}\\
 &+ C_8 \tau^{\mu+5/2}\delta^{\frac{\sigma-\mu}{\sigma+5/2}} +\tau^{\mu-\sigma}\delta^{\frac{\sigma-\mu}{\sigma+5/2}}\|\bm J \|_{\bm p, \sigma}, \quad 0\leq \mu\leq \sigma.
  \end{aligned}
  \end{equation*}
\end{rem}
\section{Numerical examples }
In this section, we carry out a series of numerical experiments to illustrate that the proposed Fourier reconstruction method is effective and efficient.

First, we briefly describe some parameters setting of our numerical experiments. Let $D=[-0.5, 0.5]^3$, namely, $a=1$. Assume that the wave propagates in the vacuum space, where $\mu_0=4\pi\times10^{-7}$ and $\varepsilon_0=8.8541\times10^{-12}$.
Synthetic electromagnetic  far-field data are generated by solving the direct problem of
\eqref{eq: Maxwell} by using the quadratic finite elements on a truncated spherical domain enclosed
by a PML layer. The mesh of the forward solver is successively refined till
the relative error of the successive measured electromagnetic wave data is below $0.1\%$.
To show the stability of our proposed method, we also add  some random noise to the synthetic far-field data by considering
\begin{equation*}
\begin{aligned}
 &\bm E_{\infty}^ \delta:=\bm  E_{\infty} +\delta r_1 |\bm  E_{\infty} |_{\infty}\mathrm{e}^{\rm{i}\pi r_2},\\
  &\bm H_{\infty}^ \delta:=\bm  H_{\infty} +\delta r_1 |\bm  H_{\infty} |_{\infty}\mathrm{e}^{\rm{i}\pi r_2},
 \end{aligned}
 \end{equation*}
where $r_1 $ and $ r_2$ are two uniform random numbers, both ranging from $-1$ to $1$, and $\delta>0$ represents the noise level.
From Remark \ref{rem:N}, the truncation $N$ is given by
\begin{equation}\label{eq:N}
  N(\delta):=[3\delta^{-2/7}]+1,
\end{equation}
where $[X]$ denotes the largest integer that is smaller than $X+1$.

Next, we specify details of obtaining the artificial multi-frequency electromagnetic far-field data.   Let
\begin{equation*}
 \mathbb{L}_{N}:= \{\bm l \in \mathbb{Z}^{3}\mid 1 \leq |\bm{l}|_\infty\leq N \},
\end{equation*}
then the wavenumber set is given by
\begin{equation*}
\mathbb{K}_N:= \left\{2\pi|\bm{l}|:\bm{l}\in  \mathbb{L}_{N}  \right \} \cup \{2\pi\lambda\}, \quad \lambda=10^{-3},
\end{equation*}
and the observation directions are given by
\begin{equation*}
 \mathbb{X}_N:= \left\{\frac{\bm l}{|\bm l|}:\bm{l}\in  \mathbb{L}_{N}  \right \} \cup \{(1,0,0)\}.
\end{equation*}
Thus,  every wavenumber and  observation direction can be denoted by $k_{j}\in \mathbb{K}_{N}$ and $ \hat{\bm x}_j \in  \mathbb{X}_N $, respectively, where $\, j=1,2,\cdots,(2N+1)^{3}$.
Correspondingly, the frequency $\omega_j$ is chosen as $\omega_j=k_j/\sqrt{\mu_0 \varepsilon_0}$. With the admissible wavenumbers defined earlier, the artificial electromagnetic far-field  data with noise can be written as
\begin{equation*}
   \left \{\left(\bm E_\infty^\delta(\hat{x}_{j};k_{j}),\bm H_\infty^\delta(\hat{x}_{j};k_{j})\right ) : \hat{\bm x}_{j}\in \mathbb{K}_{N}, k_{j}\in \mathbb{K}_{N},\, j=1,2, \cdots,(2N+1)^{3}\right \}.
\end{equation*}

 Finally, we specify details of the numerical inversion via the Fourier method. We reconstruct the electric current source  $\bm J(\bm x),\, \bm x\in D$ by the truncated Fourier expansion $\bm J_N^{\delta}(\bm x),\, \bm x\in D$,
  where
\begin{equation*}
\bm J=\begin{bmatrix}
J_{1}\\
J_{2} \\
J_{3}
\end{bmatrix}, \quad
\bm J_N^{\delta}=\begin{bmatrix}
J_{1}^{N}\\
J_{2}^{N} \\
J_{3}^{N}
\end{bmatrix}.
\end{equation*}
Given the  noisy far-field data defined above,
if we use the electric far-field data $\{\bm E_\infty^\delta(\hat{x}_{j};k_{j})\}$, then
the Fourier coefficients $\hat{f}_{\bm l}, \hat{g}_{\bm l}, 1\leq |\bm l|_\infty \leq N $  and $\hat{f}_{\bm 0}$ are computed by   \eqref{eq:Ef_hat}, \eqref{eq:Eg_hat} and \eqref{eq:Ef0_hat}, respectively. If we use the magnetic far-field data $\{\bm H_\infty^\delta(\hat{x}_{j};k_{j})\}$, then
the Fourier coefficients $\hat{f}_{\bm l}, \hat{g}_{\bm l}, 1\leq |\bm l|_\infty \leq N $  and $\hat{f}_{\bm 0}$ are computed by   \eqref{eq:Hf_hat}, \eqref{eq:Hg_hat} and \eqref{eq:Hf0_hat}, respectively.
Divide the  domain $D$ into a mesh with  a uniform grid of size $ 50\times 50\times 50  $. The approximated  Fourier series $\bm J_N^{\delta}(\bm z)$ are computed at the mesh nodes $\bm z_j,\, j=1,2,  \cdots, 50^3$ by  \eqref{eq:J_N}. The relative error is defined as
 \begin{equation*}
   \mathrm{relative \ error}=\frac{\|\bm J-\bm J_N^{\delta}\|_{L^2(D)}}
   {\|\bm J\|_{L^2(D)}}.
 \end{equation*}
Unless specified otherwise, we use the magnetic far-field data to reconstruct the electric current source.

Based on the above discussion, we formulate the reconstruction scheme by the Fourier method in Algorithm S as follows.

\begin{table}[htp]
\centering
\begin{tabular}{cp{.8\textwidth}}
\toprule
\multicolumn{2}{l}{{\bf Algorithm S:}\quad Fourier method for reconstructing  the electromagnetic source }\\
\midrule
 {\bf Step 1} & Choose the parameters $\lambda$, $N$, the wavenumber set $\mathbb{K}_N$ and observation direction set $\mathbb{X}_N$.  \\
{\bf Step 2} & Collect the measured electric far-field data $\bm E_\infty^\delta(\hat{x}_{j};k_{j})$  or the magnetic far-field data $\bm H_\infty^\delta(\hat{x}_{j};k_{j})$  for $\hat{x}_{j}\in \mathbb{X}_N$ and $k_{j}\in \mathbb{K}_N $.\\
{\bf Step 3} & Compute the Fourier coefficients $\hat{f}_{\bm 0}$,  $\hat{f}_{\bm l}$ and $\hat{g}_{\bm l}$ for $1\leq |\bm l|_{\infty}\leq N$. \\
{\bf Step 4} &  Select a sampling mesh  $\mathcal{T}_h$  in a region $D$. For each sampling point $z_j\in \mathcal{T}_h $, calculate the imaging functional $\bm J_N$ defined in $\eqref{eq:J_N}$, then $\bm J_N$ is the  reconstruction of $\bm J$. \\
\bottomrule
\end{tabular}
\end{table}

\begin{example}\label{example1}
 In this example, we numerically estimate the stability of the proposed method. We consider the following smooth source function
 \begin{equation*}
\bm J=\bm p\times \nabla g,
 \end{equation*}
where
\begin{equation*}
  \begin{aligned}
   &\bm p=\frac{1}{4}\left(\sqrt{5}, -2, \sqrt{7} \right),\\
   &g(x_1, x_2, x_3)=10\left(x_1^2+x_2^2\right)\exp\left(-50\left(x_1^2+x_2^2+x_3^2\right)\right).
  \end{aligned}
\end{equation*}
\end{example}

\begin{figure}
\hfill\subfigure[]{\includegraphics[width=0.32\textwidth]
                   {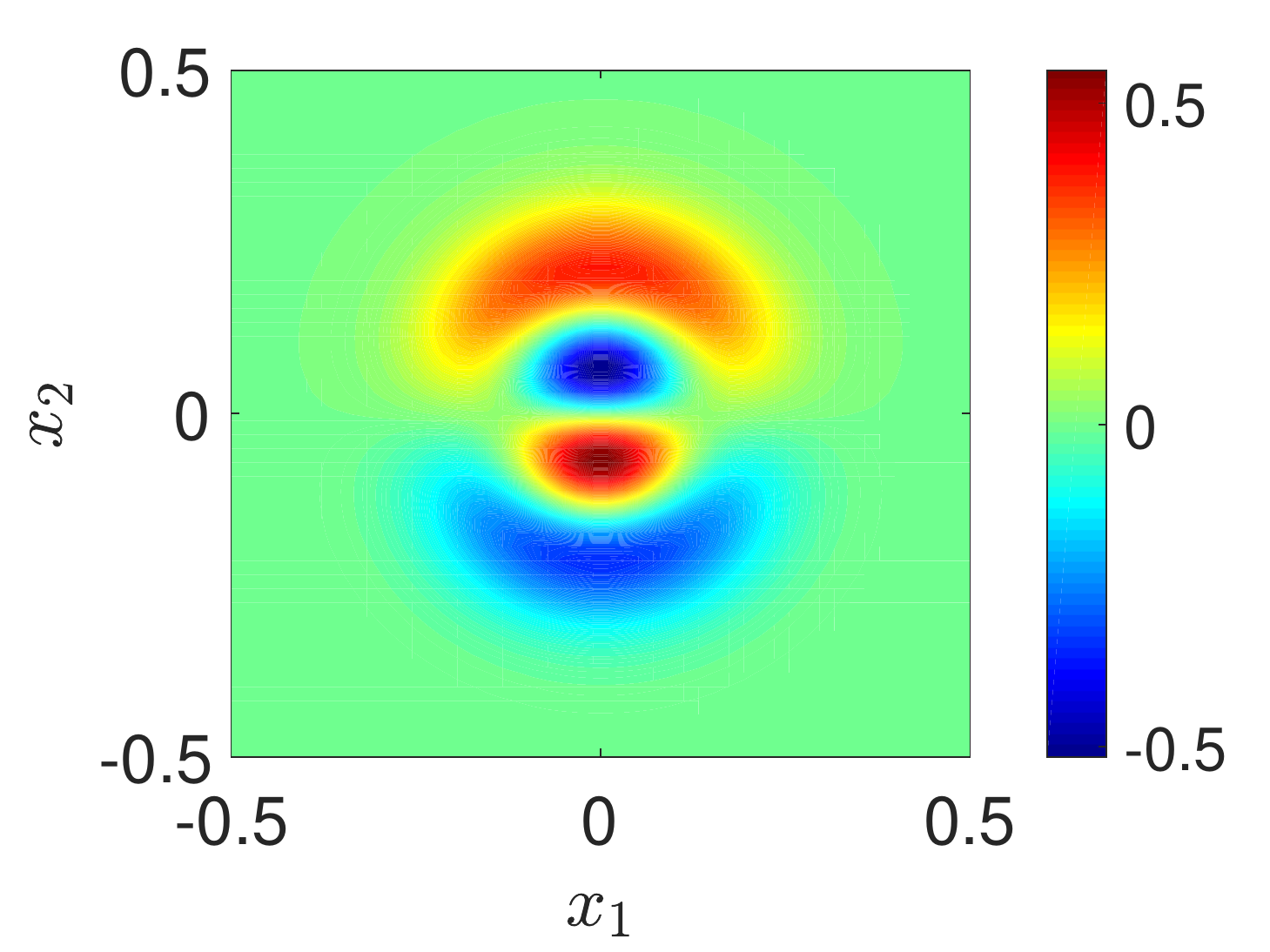}}\hfill
\hfill\subfigure[]{\includegraphics[width=0.32\textwidth]
                   {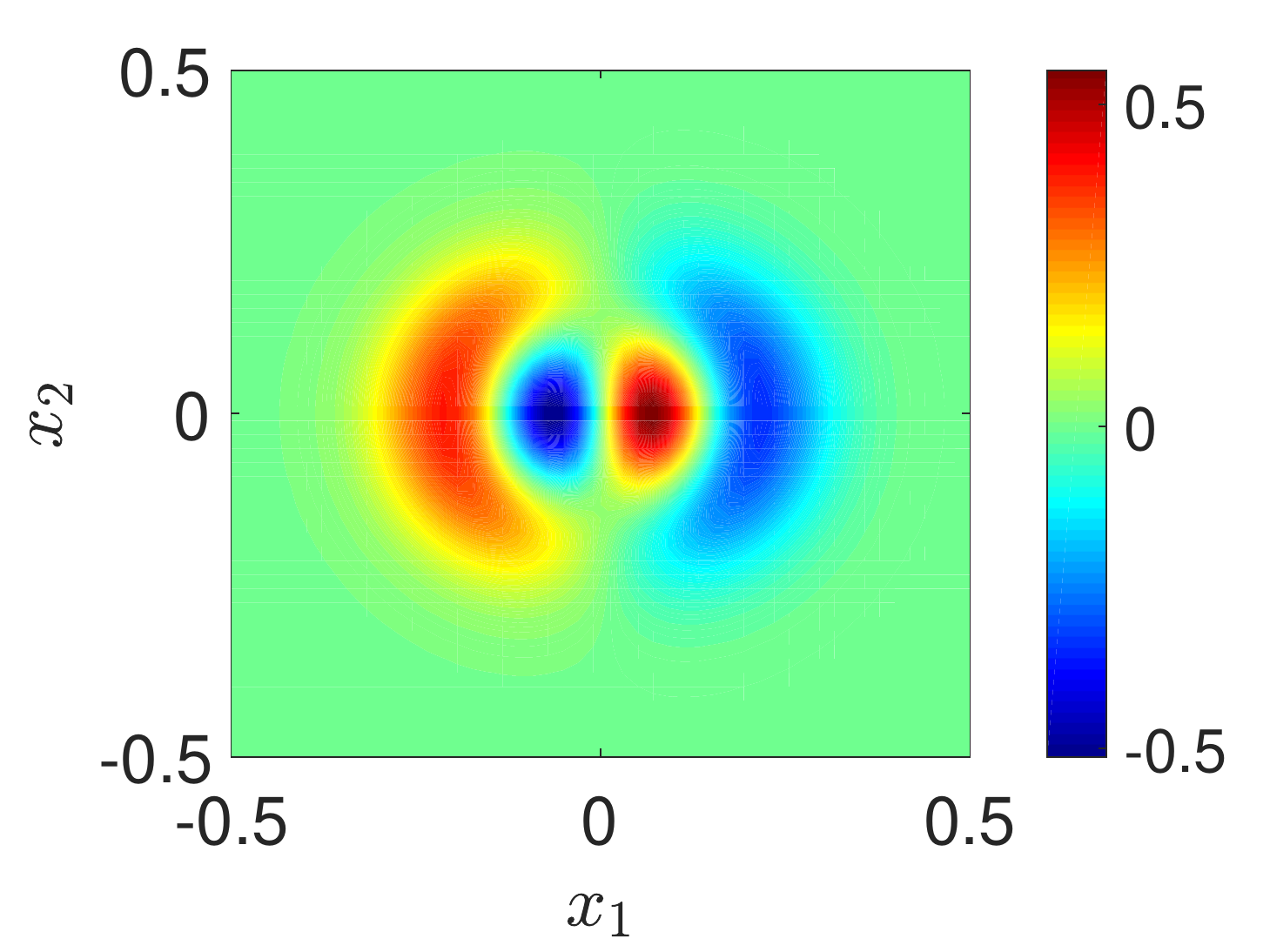}} \hfill
\hfill\subfigure[]{\includegraphics[width=0.32\textwidth]
                   {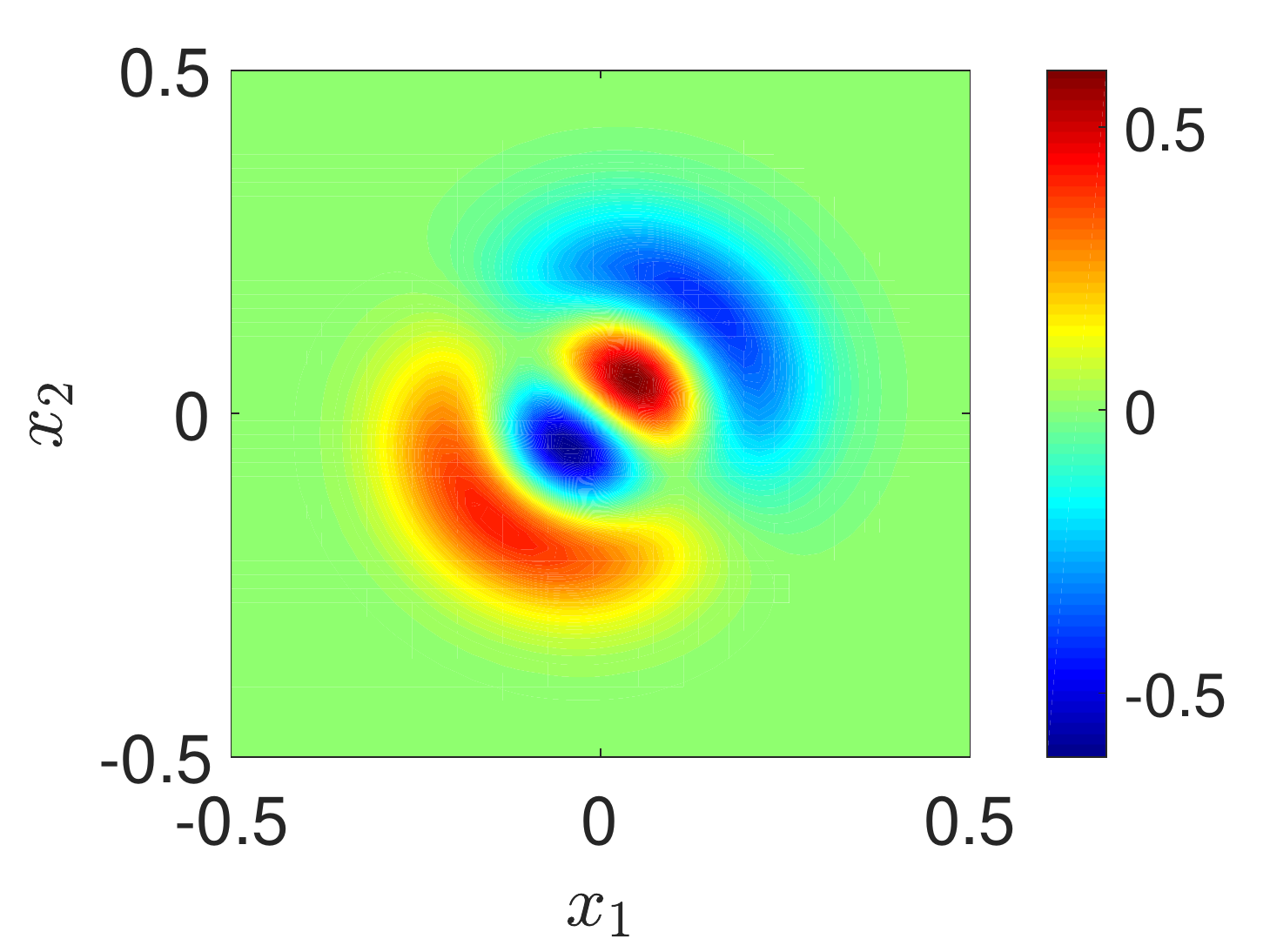}} \hfill\\
\hfill\subfigure[]{\includegraphics[width=0.32\textwidth]
                   {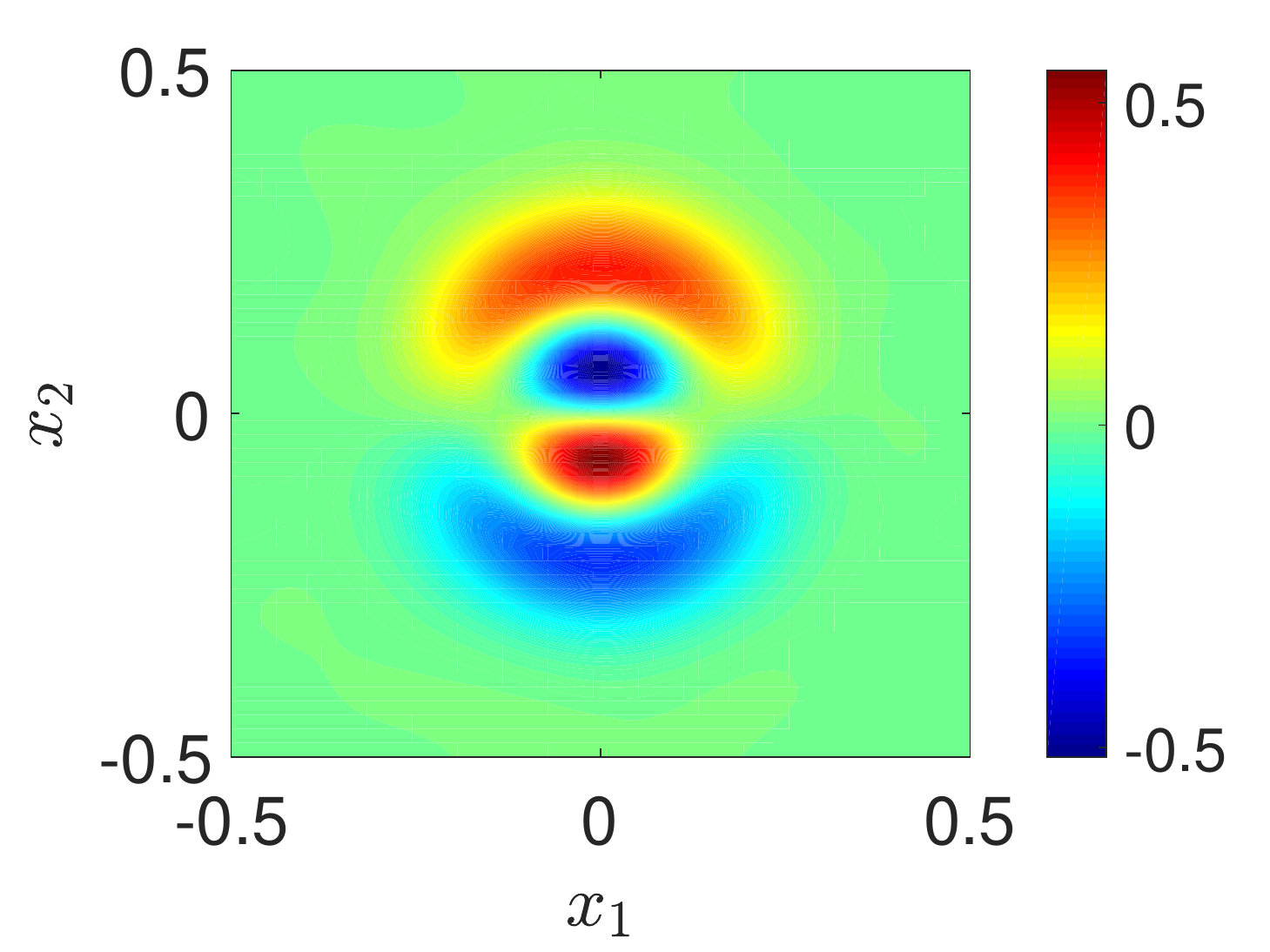}}\hfill
\hfill\subfigure[]{\includegraphics[width=0.32\textwidth]
                   {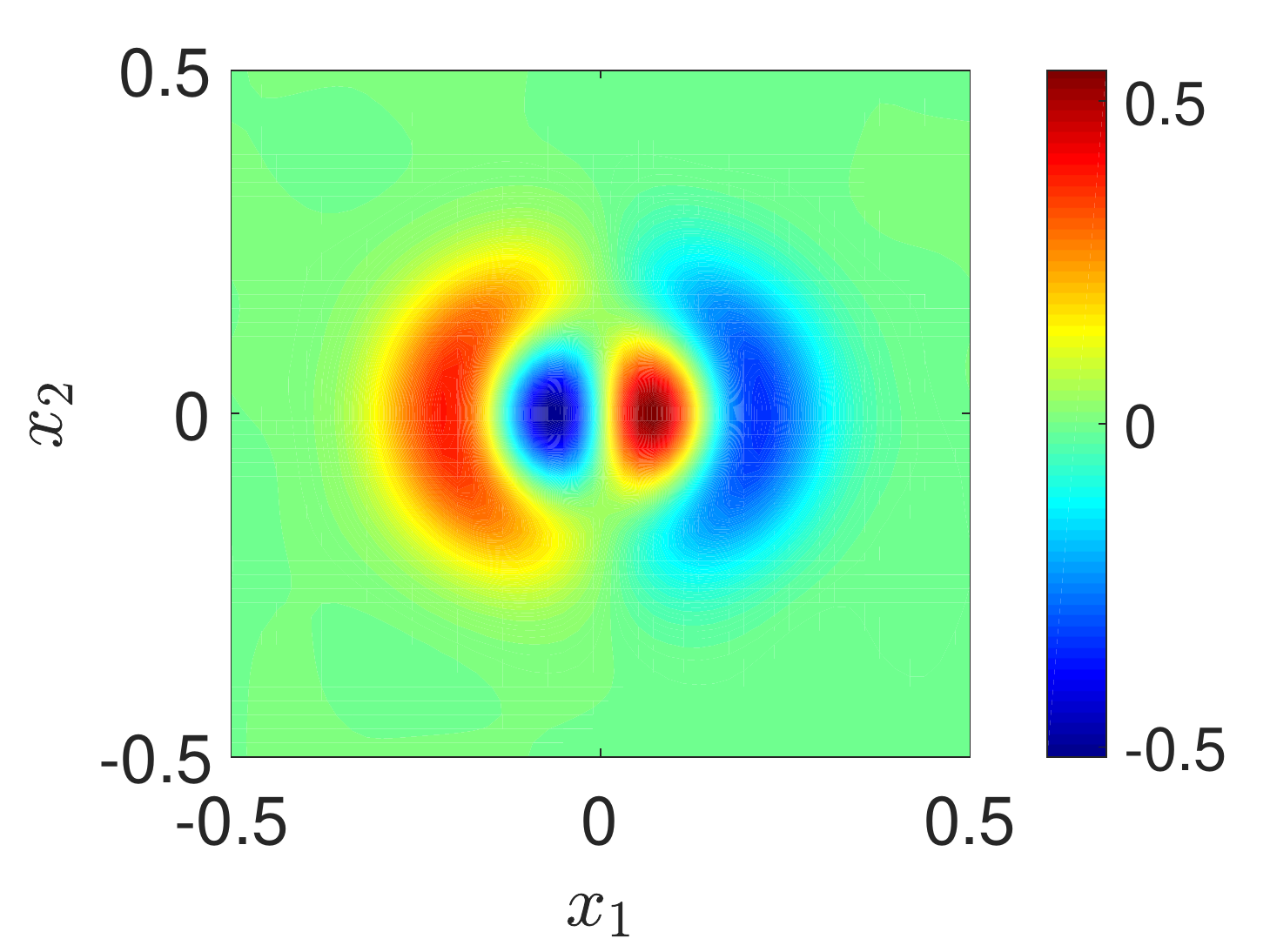}} \hfill
\hfill\subfigure[]{\includegraphics[width=0.32\textwidth]
                   {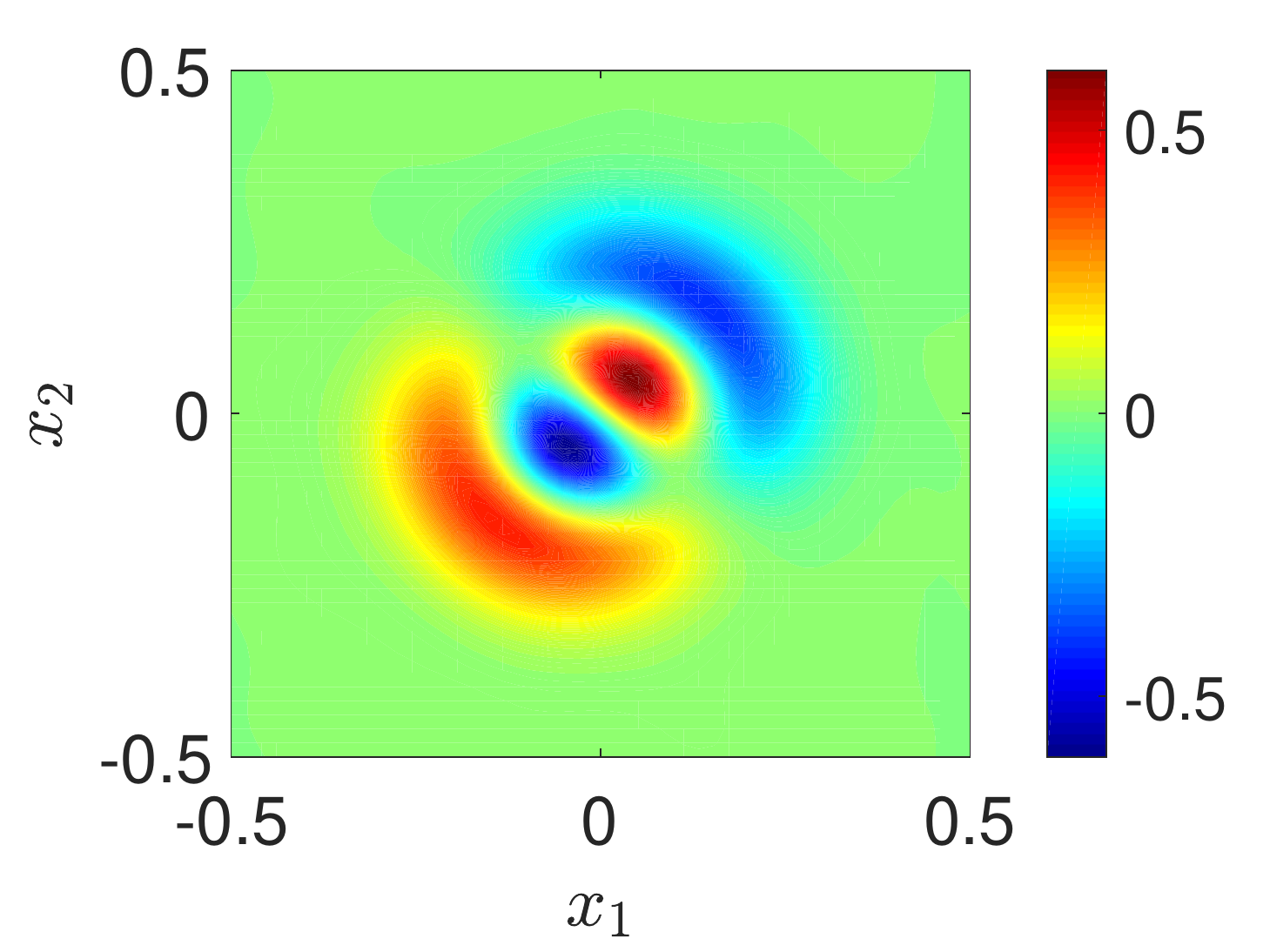}} \hfill
\caption{\label{fig:2} Contour plots of the  exact  and reconstructed  source function of Example \ref{example1} at the plane $x_3=0$, where $\delta=2\%$. (a) $J_1$,  (b) $J_2$,  (c) $J_3$, (d) $J_1^{10}$, (e) $J_2^{10}$, (f) $ J_3^{10}$.}
\end{figure}

\begin{table}[h]
\center
 \caption{ The relative errors of the reconstructions  with different noise levels $\delta$. }\label{tab1}
 \begin{tabular}{lcccccc}
  \toprule
 & $ \delta $        & 2\%     & 5\%   & 10\%    &20\%    \\
  \midrule
& $ N(\delta) $                & 10      & 8     & 6       &5 \\
& Relative error       & 0.10\%   &2.10\% & 4.26\%  & 8.94\%   \\
& Time (second)          & 78    & 40   & 12        &11   \\
  \bottomrule
 \end{tabular}
\end{table}
Figure \ref{fig:2} shows the comparison between the exact and the reconstructed source function at the plane $x_3=0$ with the additional noise  $\delta=2\%$. We observe that
the reconstructions are very close to the exact one. To exhibit the accuracy quantitatively, we list the relative errors in $L^2$  in table \ref{tab1}. Meanwhile, table \ref{tab1} illustrates that the stability and CPU time increase as the truncation order $N(\delta)$ increases.

\begin{example}\label{example2}
 In this example, we use the electric far-field data to recover the source.  We aim to recover a smooth source as follows
 \begin{equation*}
\bm J=\bm p f + \bm p\times \nabla g;
 \end{equation*}
where
\begin{equation*}
  \begin{aligned}
   &\bm p= \frac{1}{3}\left(\sqrt{5}, -1, \sqrt{3}\right),\\
   & f(x_1, x_2, x_3) =3\exp\left(-80\left((x_1-0.15)^2+(x_2-0.15)^2+x_3^2\right)\right),\\
   &g(x_1, x_2, x_3) =0.3\exp\left(-40\left(x_1^2+x_2^2+x_3^2\right)\right).
  \end{aligned}
\end{equation*}
\end{example}
\begin{figure}
\hfill\subfigure[]{\includegraphics[width=0.32\textwidth]
                   {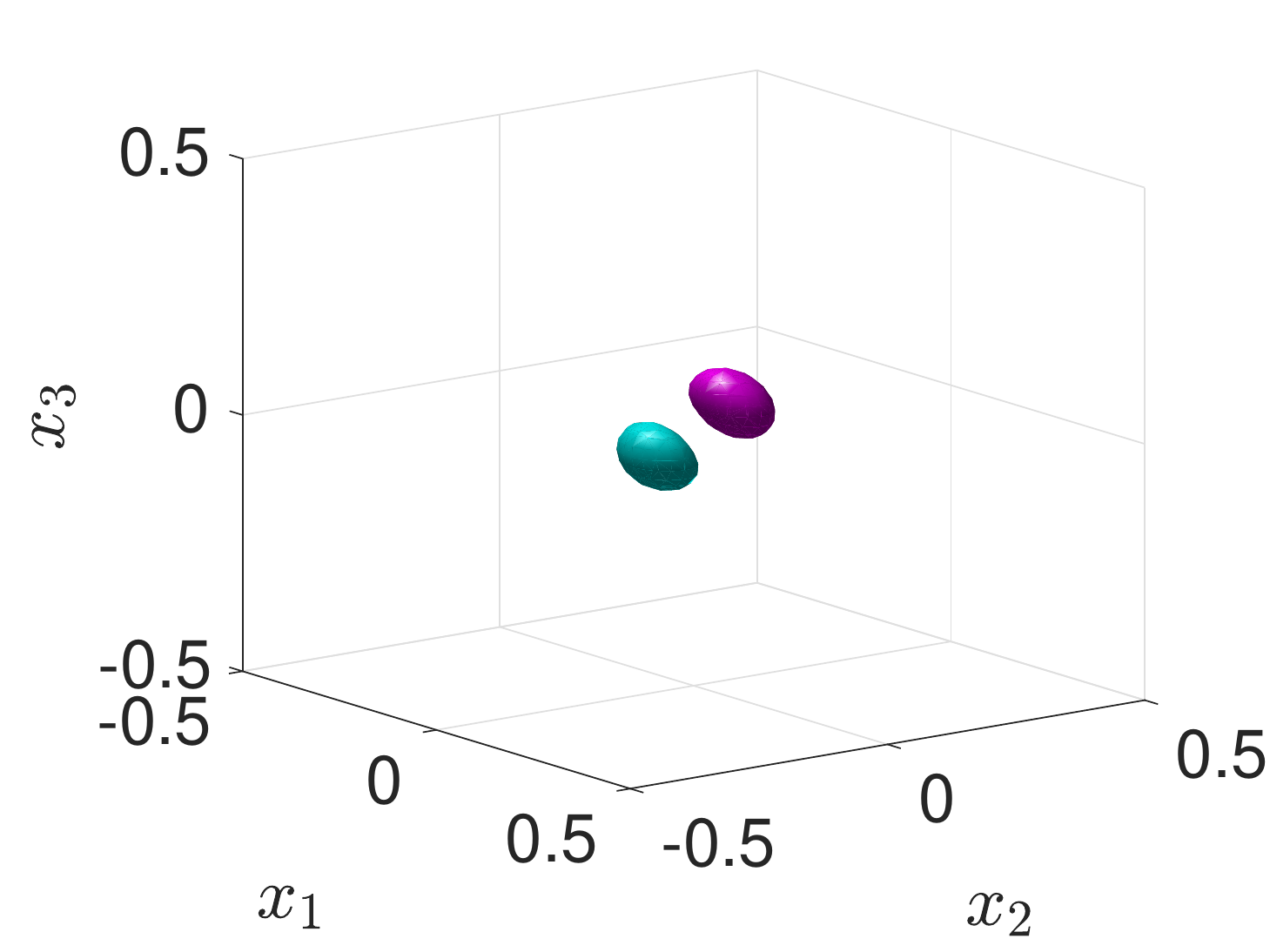}}\hfill
\hfill\subfigure[]{\includegraphics[width=0.32\textwidth]
                   {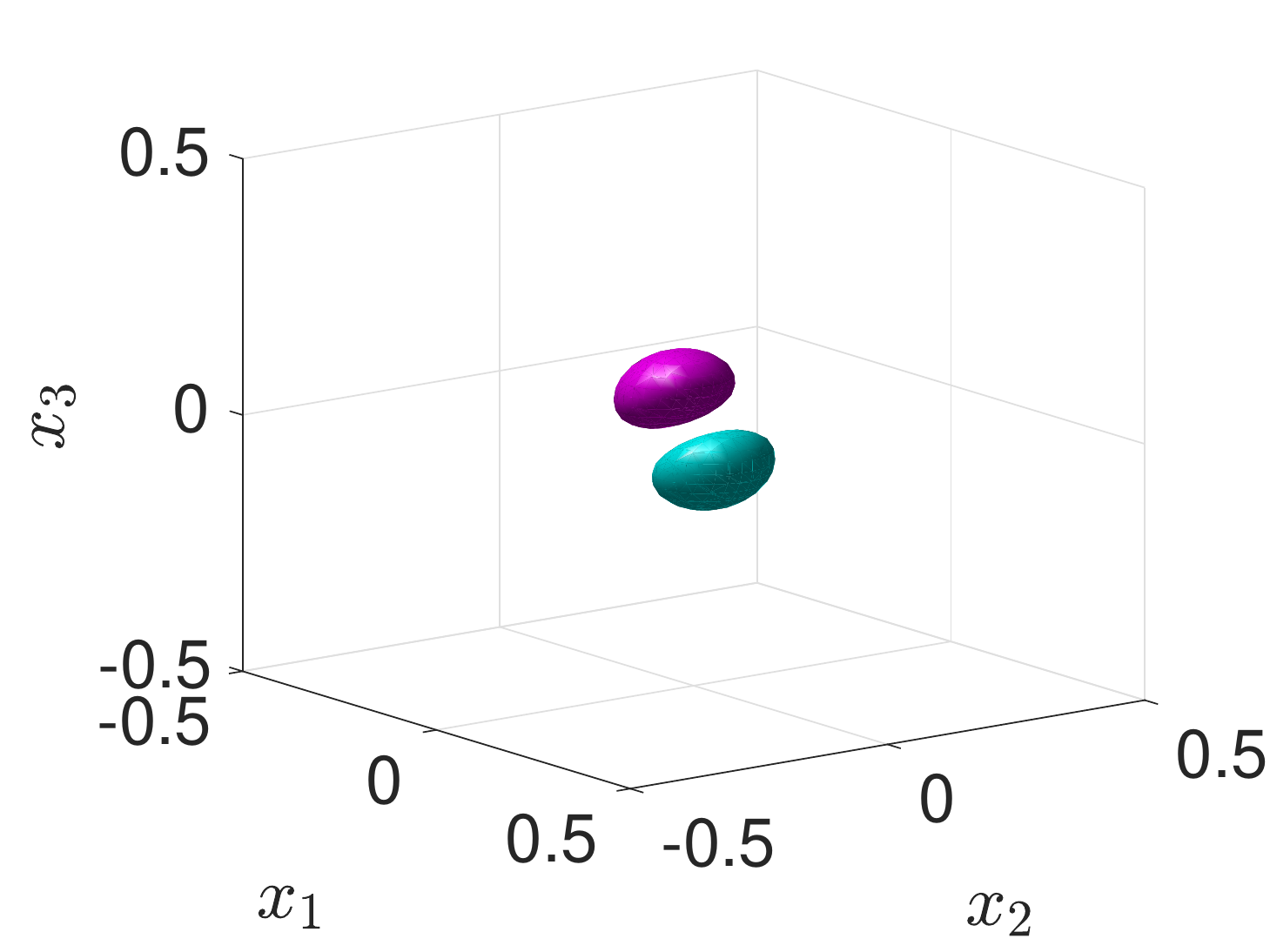}} \hfill
\hfill\subfigure[]{\includegraphics[width=0.32\textwidth]
                   {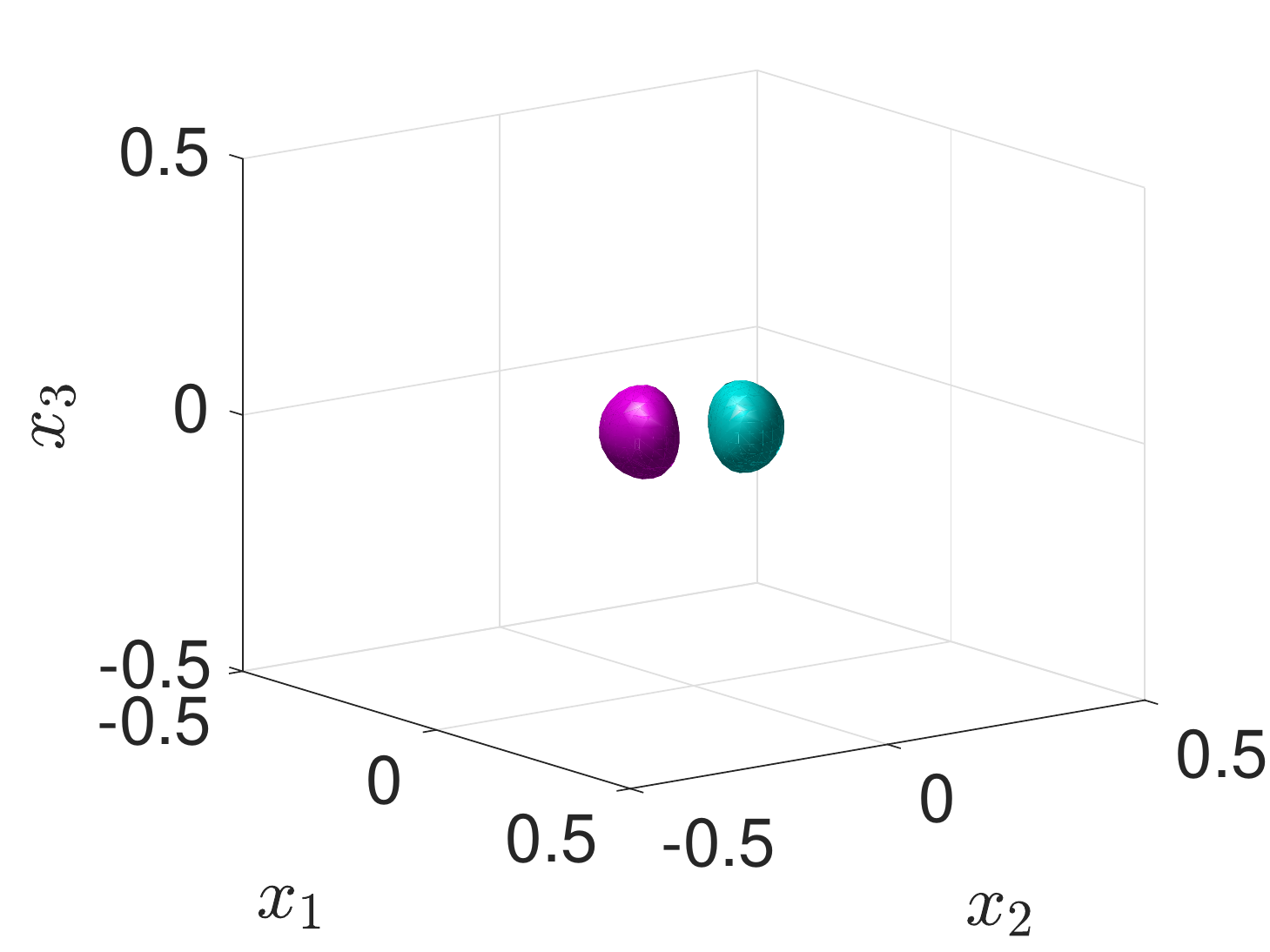}} \hfill\\
\hfill\subfigure[]{\includegraphics[width=0.32\textwidth]
                   {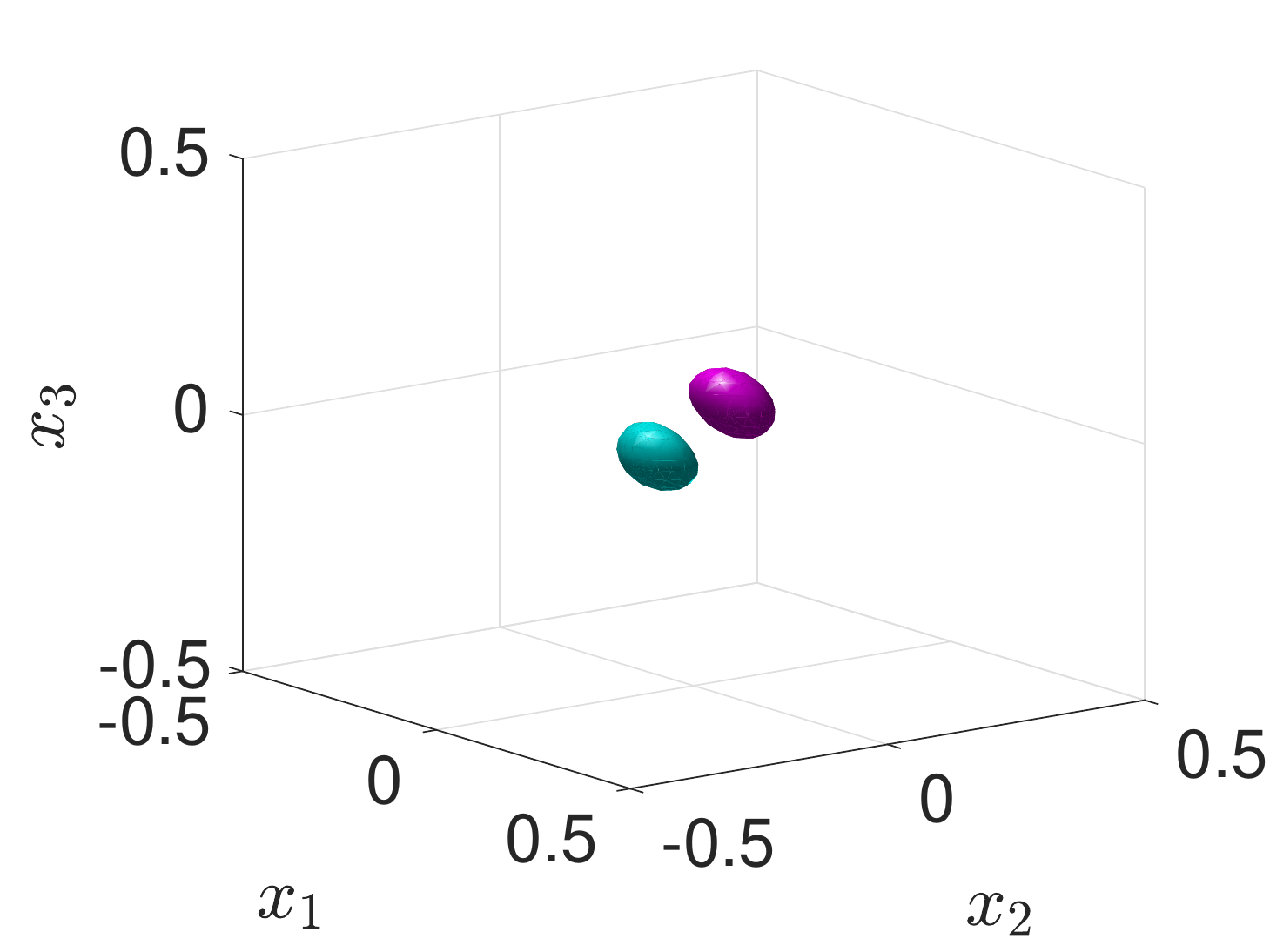}}\hfill
\hfill\subfigure[]{\includegraphics[width=0.32\textwidth]
                   {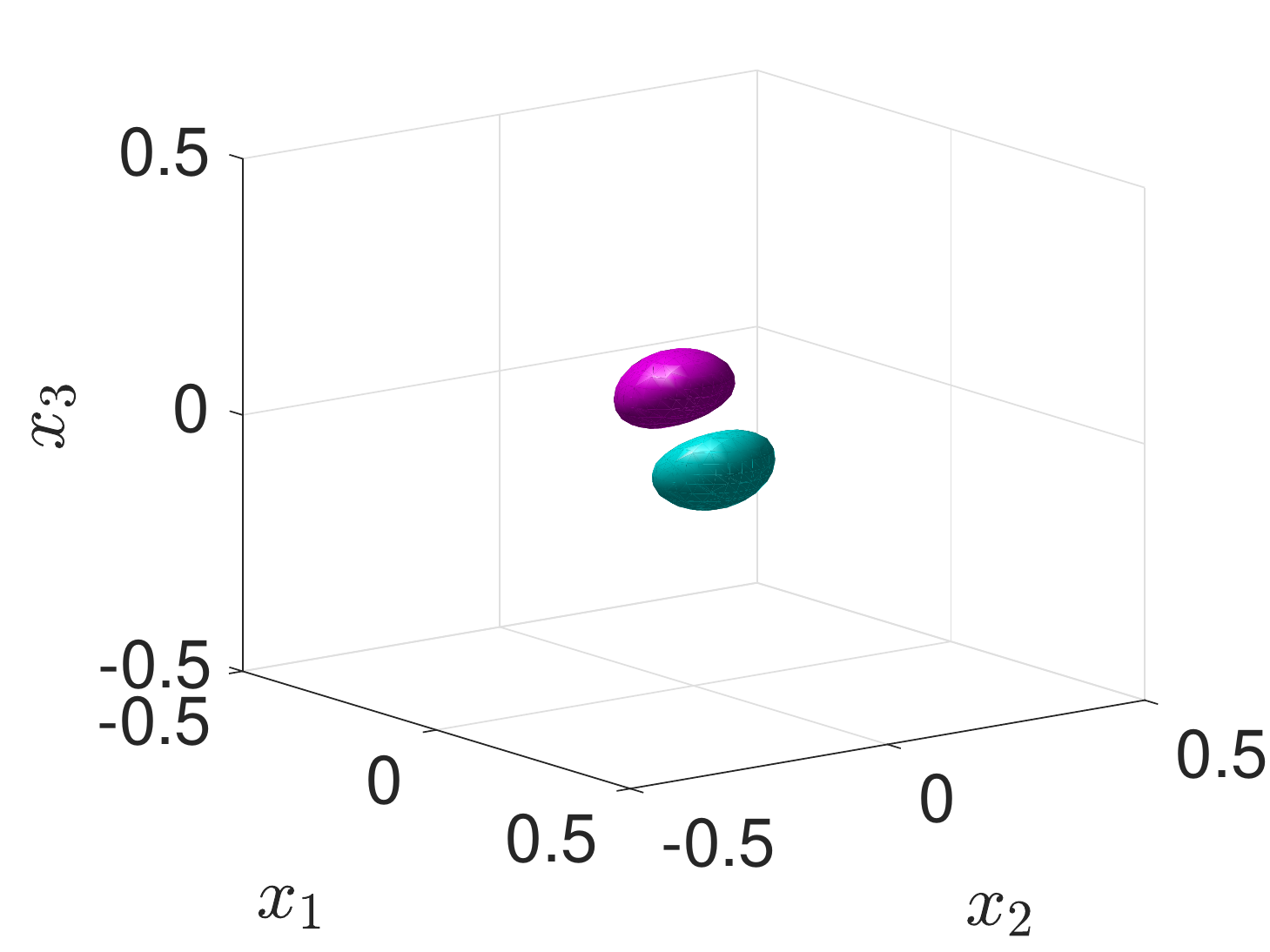}} \hfill
\hfill\subfigure[]{\includegraphics[width=0.32\textwidth]
                   {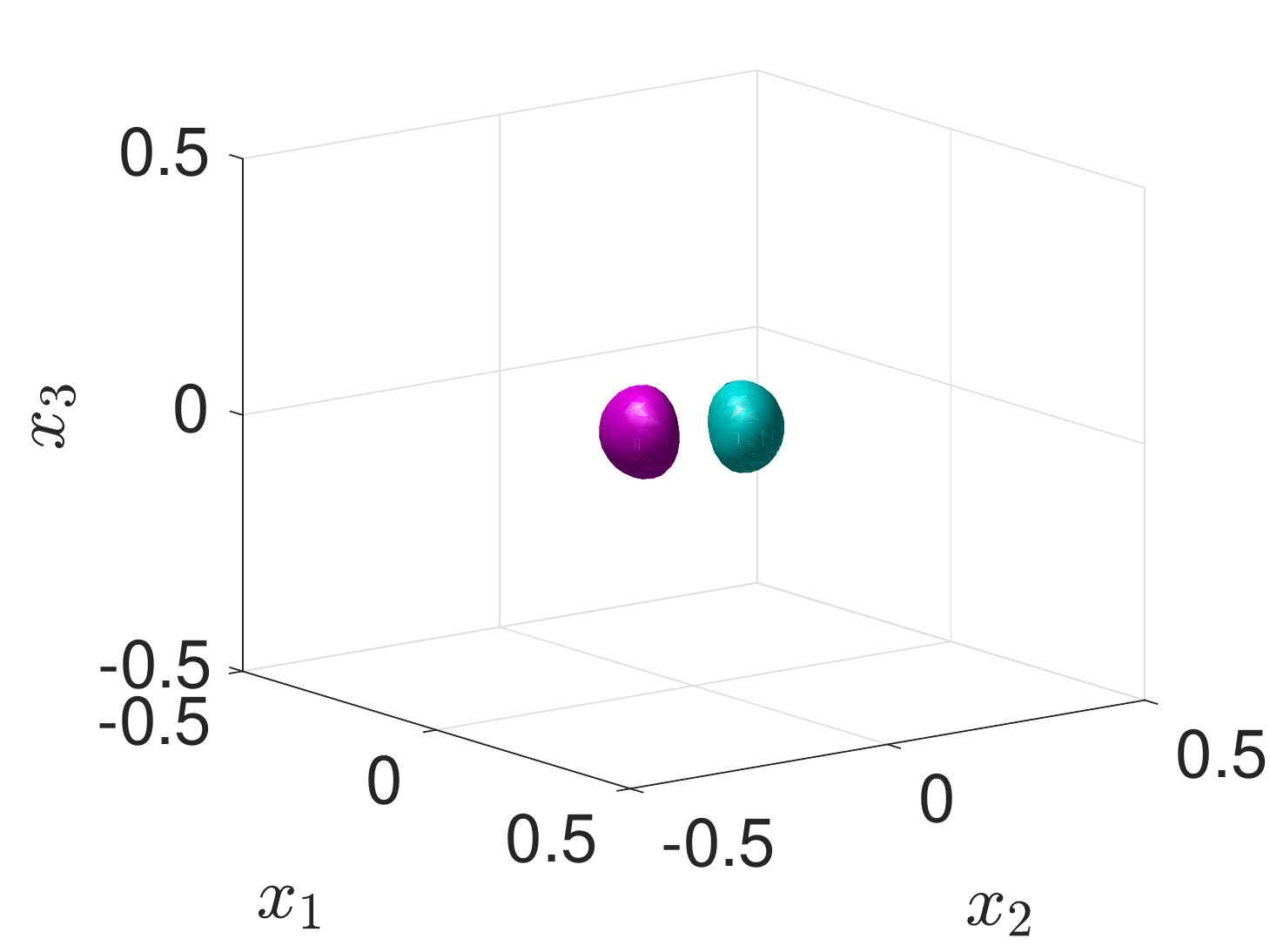}} \hfill
\caption{\label{fig:3} Iso-surface plots of the  exact  and the reconstructed vectorial source function of Example \ref{example2}, where the red color denotes the iso-surface level being $10$ and the green  color denotes iso-surface level being  $ -10$. (a) $J_1$,  (b) $J_2$,  (c) $J_3$, (d) $J_1^{10}$, (e) $J_2^{10}$, (f) $ J_3^{10}$.}
\end{figure}
Figure \ref{fig:3} presents the iso-surface plots of the exact source and the reconstruction with noise $2\%$ , which demonstrate clearly that our proposed method performance nicely.

\begin{example}\label{example3}
 In this example, we consider a discontinuous source function. For simplicity, the source function is given by
 \begin{equation*}
\bm J=\bm p f,
 \end{equation*}
where
\begin{equation*}
  \begin{aligned}
   &\bm p=\frac{1}{\sqrt{6}}\left(1, \sqrt{2}, \sqrt{3}\right),\\
   & \displaystyle f(x_1, x_2, x_3) =
   \begin{cases}
  &\displaystyle 1,          \quad \mathrm{if } \ (x_1+0.25)^2+x_2^2+x_3^2\leq 0.15^2,\medskip\\
  & \displaystyle\frac{1}{2}, \quad \mathrm{if}\  0.1\leq x_1\leq 0.4, -0.15\leq x_2\leq 0.15,-0.15\leq x_3\leq 0.15,\medskip\\
  & \displaystyle 0,   \quad \mathrm{elsewhere}.
   \end{cases}
  \end{aligned}
\end{equation*}
\end{example}

\begin{figure}
\hfill\subfigure[]{\includegraphics[width=0.32\textwidth]
                   {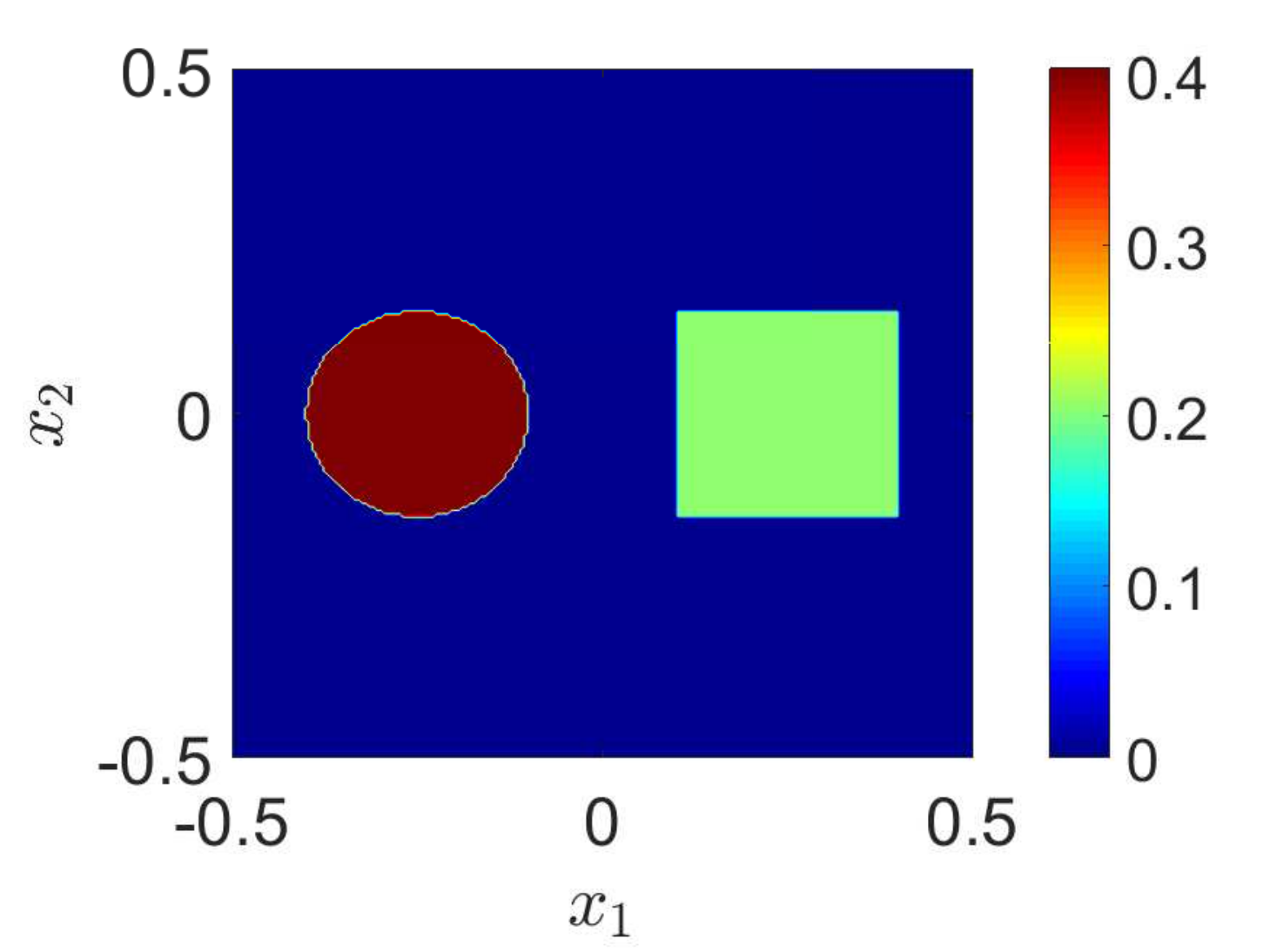}}\hfill
\hfill\subfigure[]{\includegraphics[width=0.32\textwidth]
                   {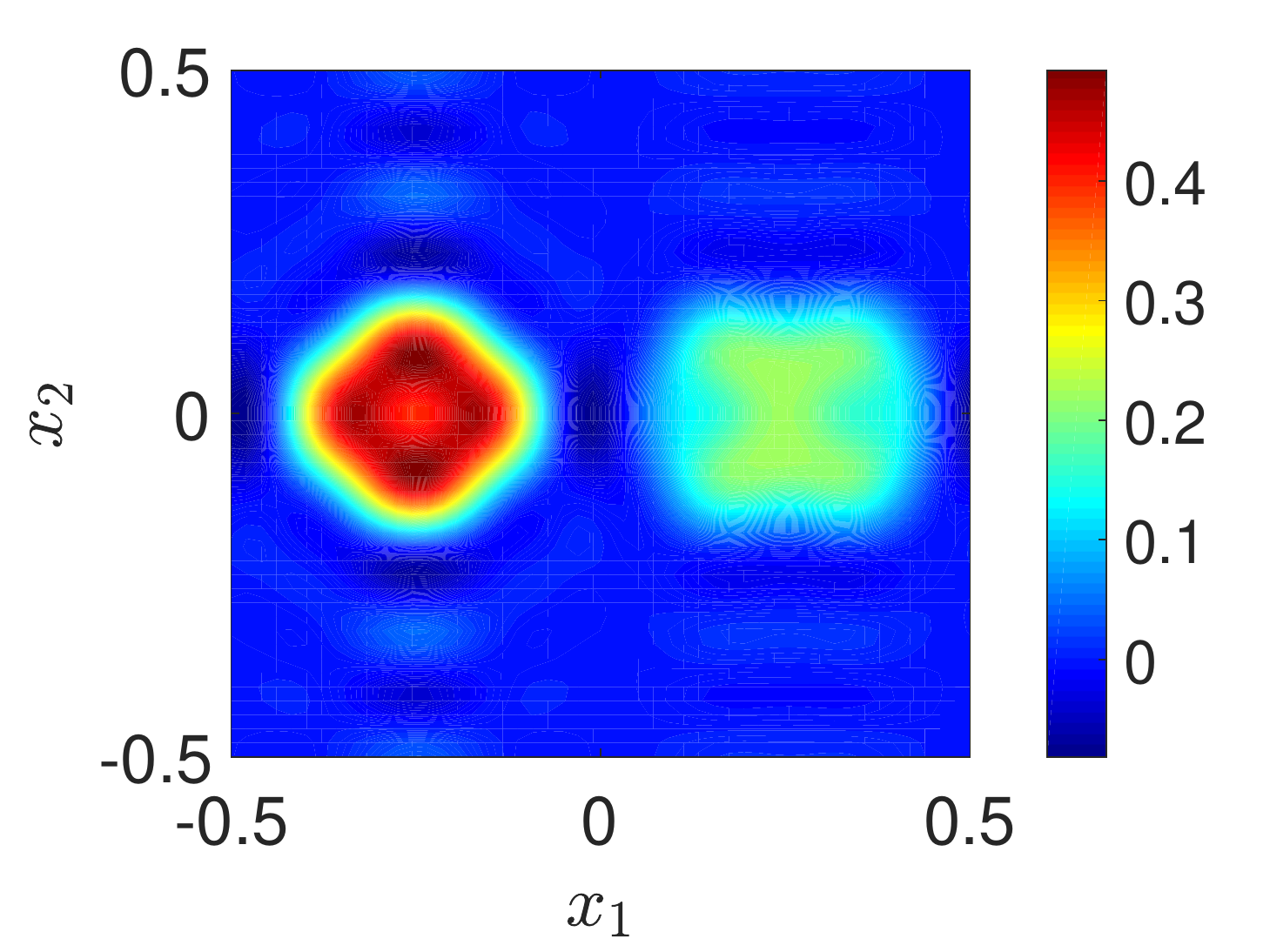}} \hfill
\hfill\subfigure[]{\includegraphics[width=0.32\textwidth]
                   {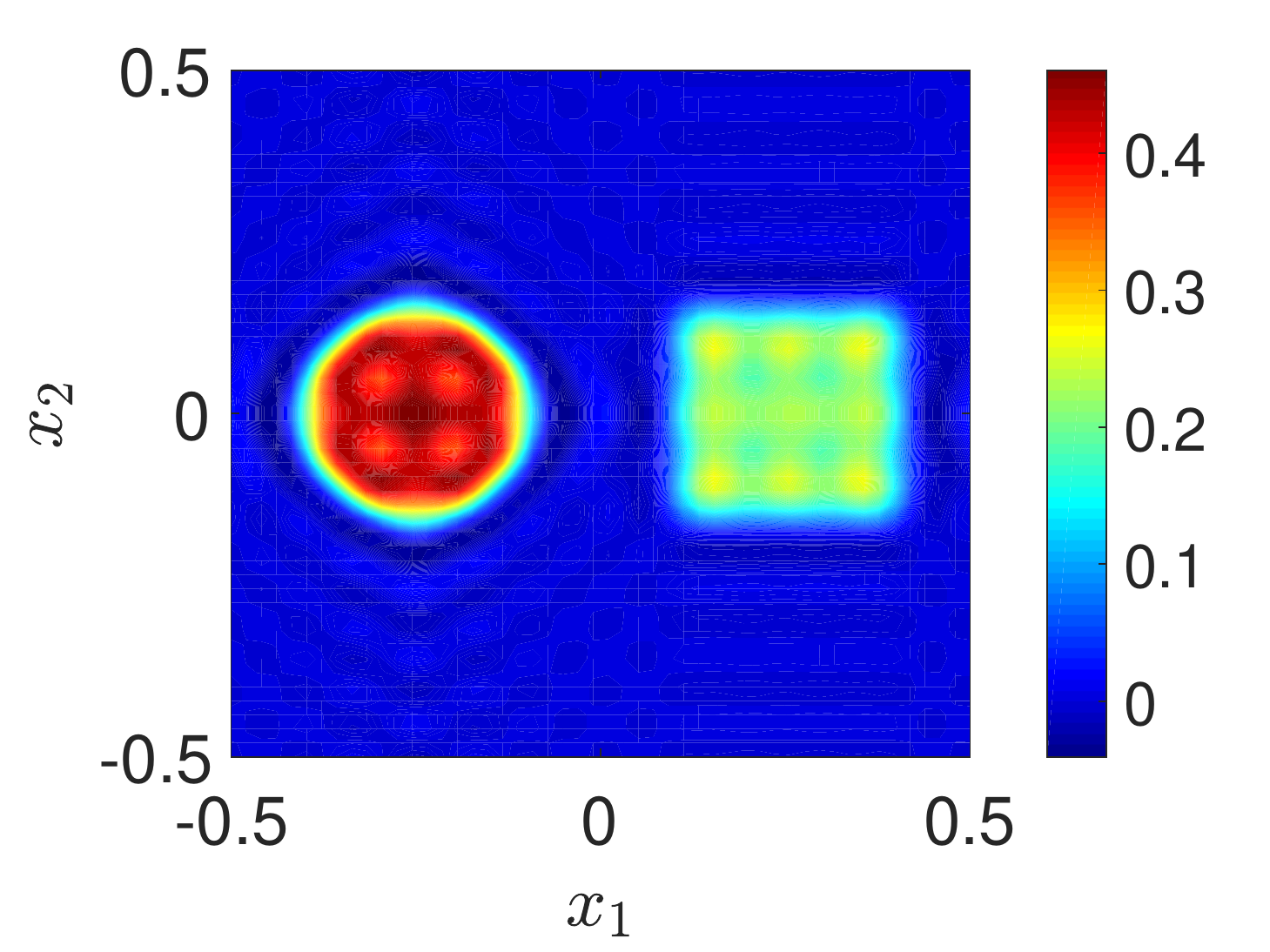}} \hfill\\
\hfill\subfigure[]{\includegraphics[width=0.32\textwidth]
                   {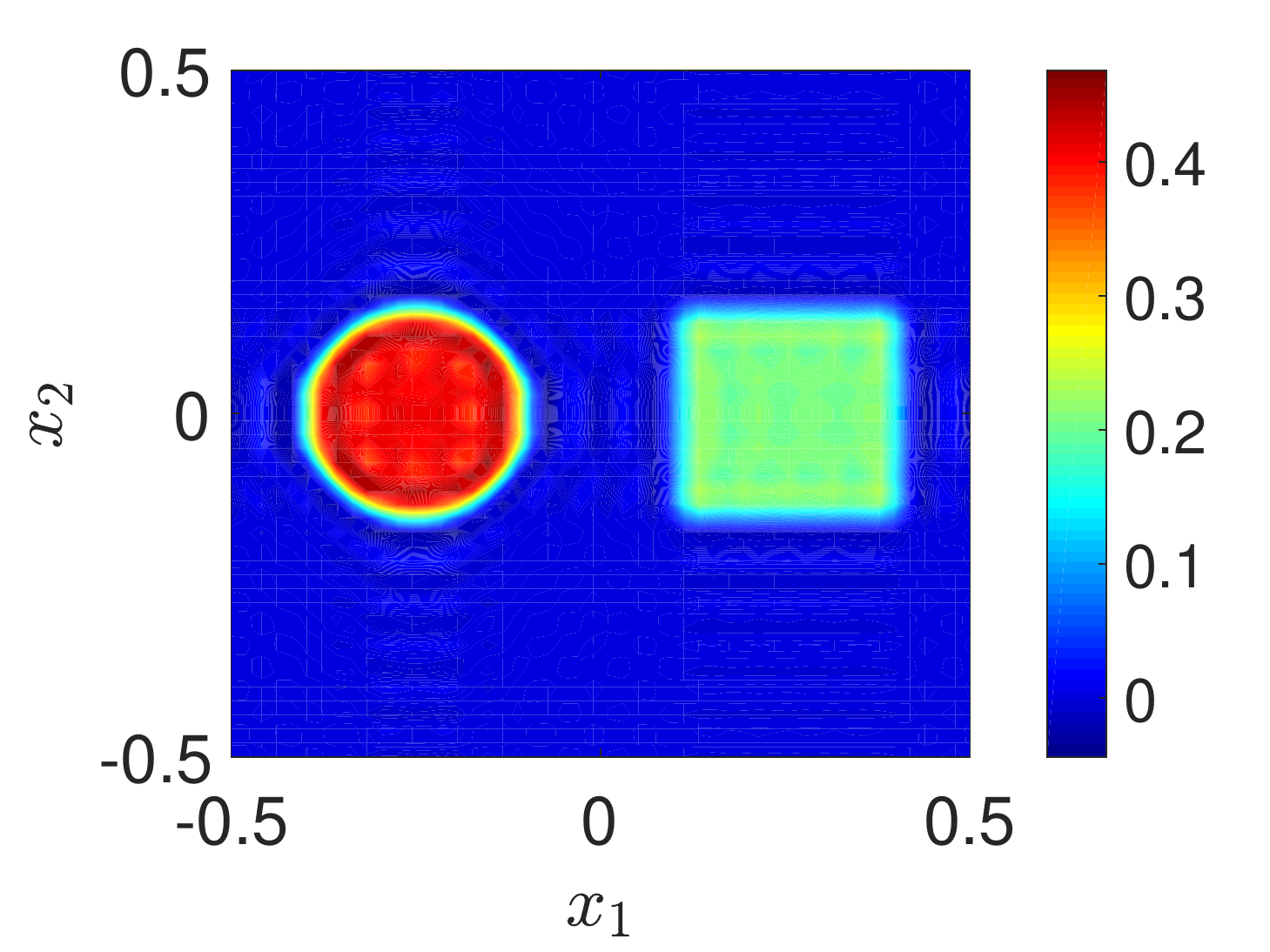}}\hfill
\hfill\subfigure[]{\includegraphics[width=0.32\textwidth]
                   {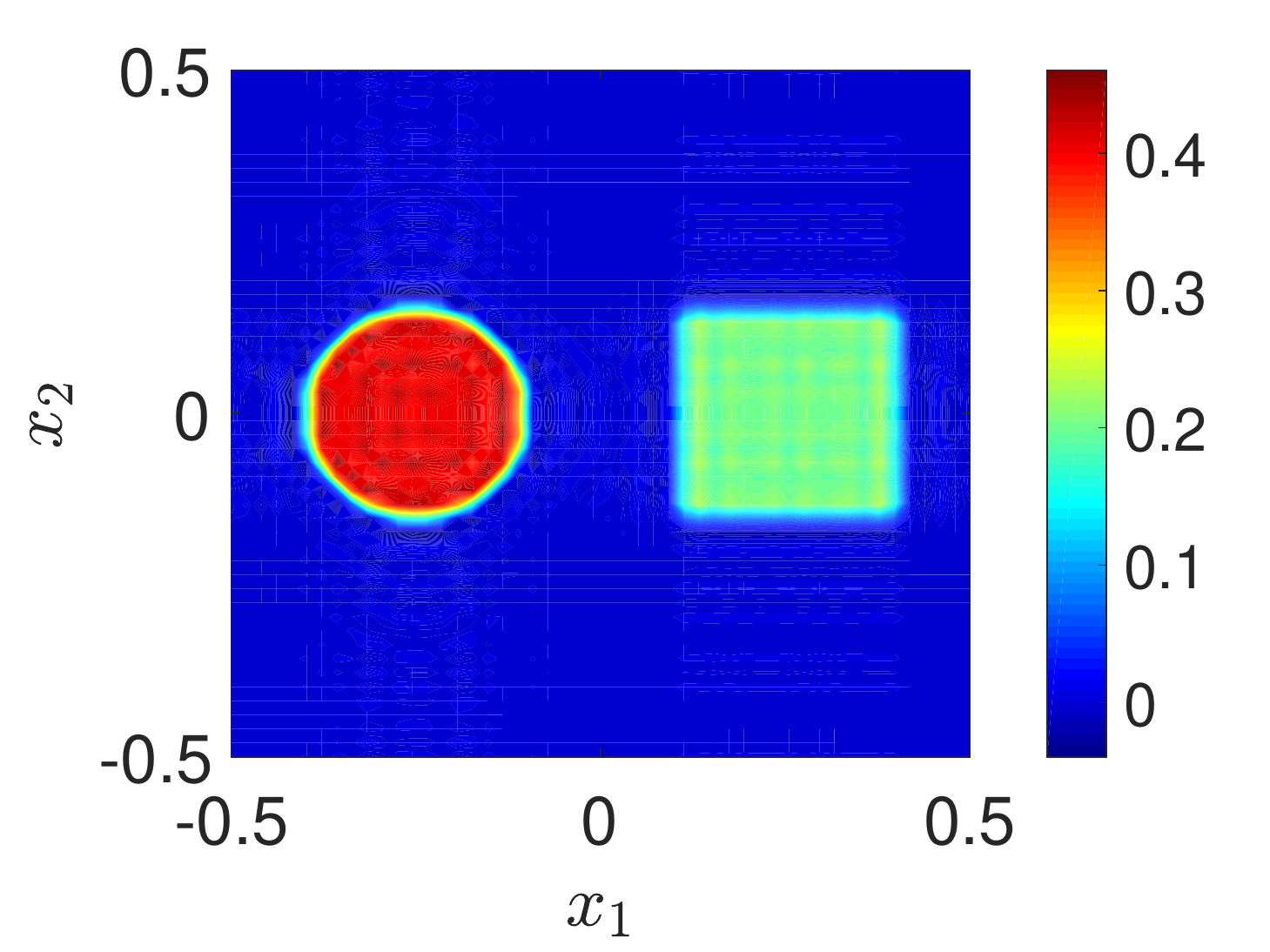}} \hfill
\hfill\subfigure[]{\includegraphics[width=0.32\textwidth]
                   {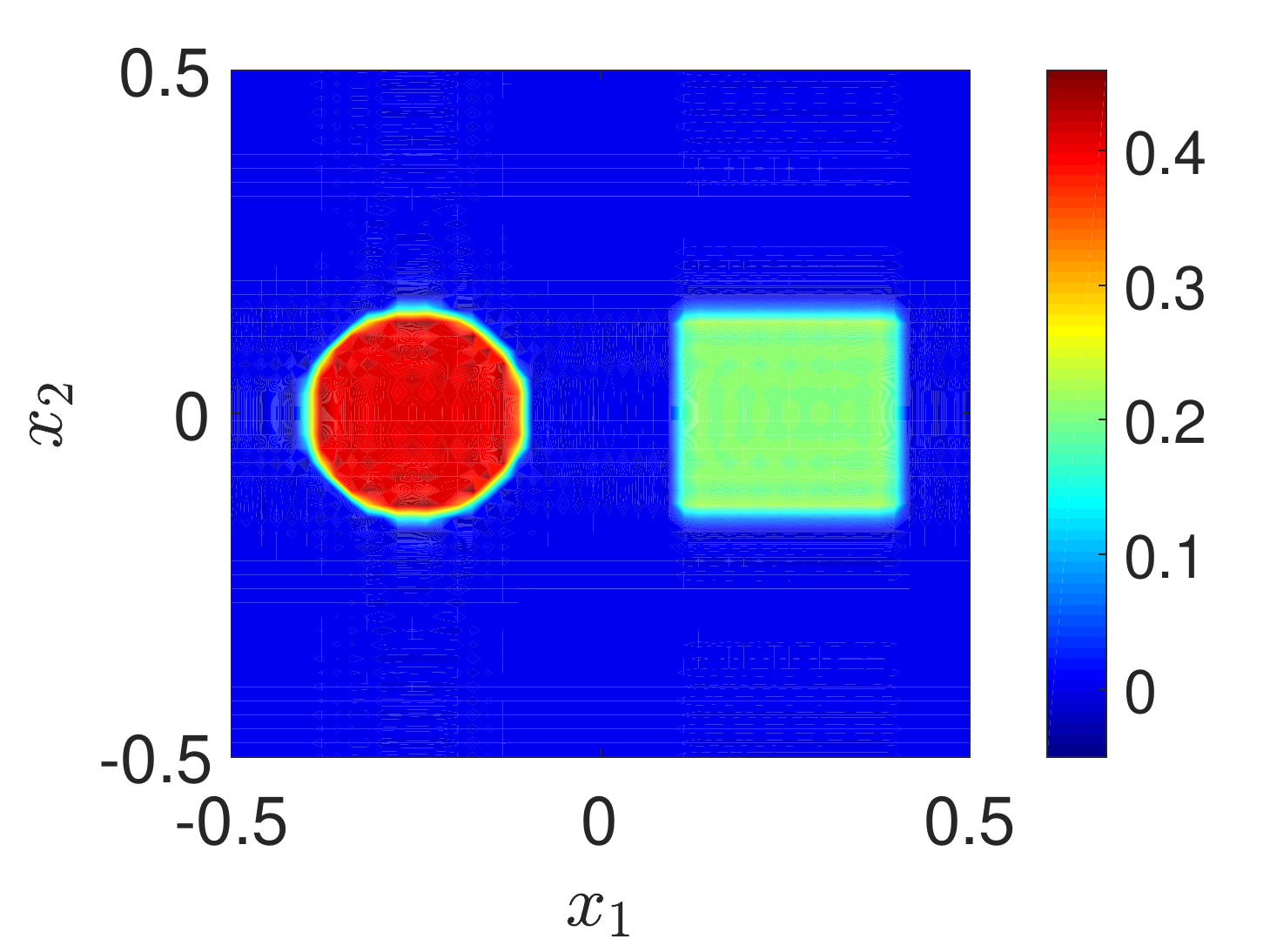}} \hfill
\caption{\label{fig:4} Contour plots of the  exact  and the reconstructed vector source function in Example \ref{example3} at the plane $x_3=0$. (a) exact $J_1$,  (b) $J_1^5$,  (c)  $J_1^{10}$, (d)  $J_1^{15}$, (e)  $J_1^{20}$, (f)  $J_1^{25}$.}
\end{figure}

\begin{figure}
\centering
\hfill\subfigure[]{\includegraphics[width=0.32\textwidth]
                   {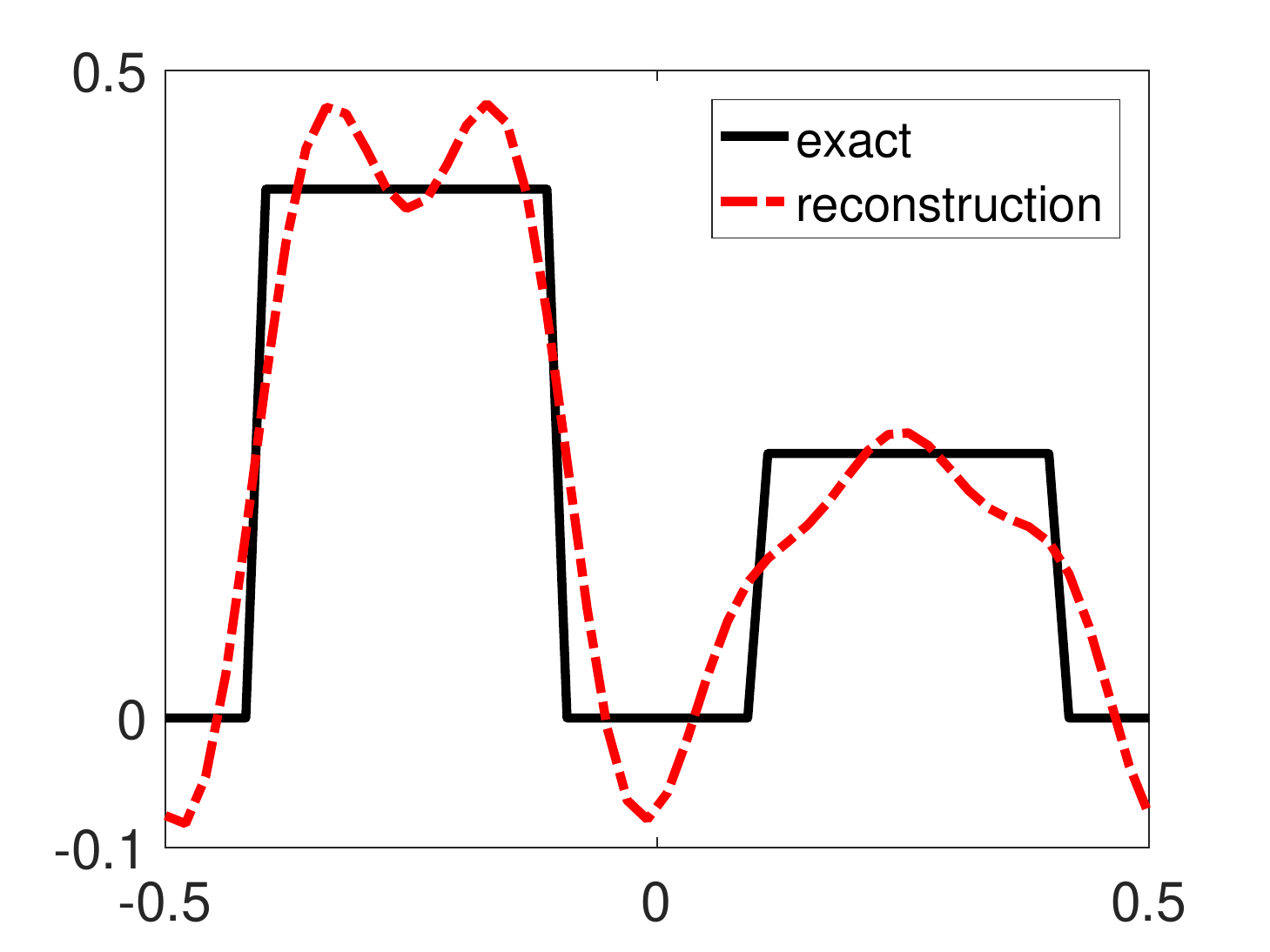}}\hfill
\hfill\subfigure[]{\includegraphics[width=0.32\textwidth]
                   {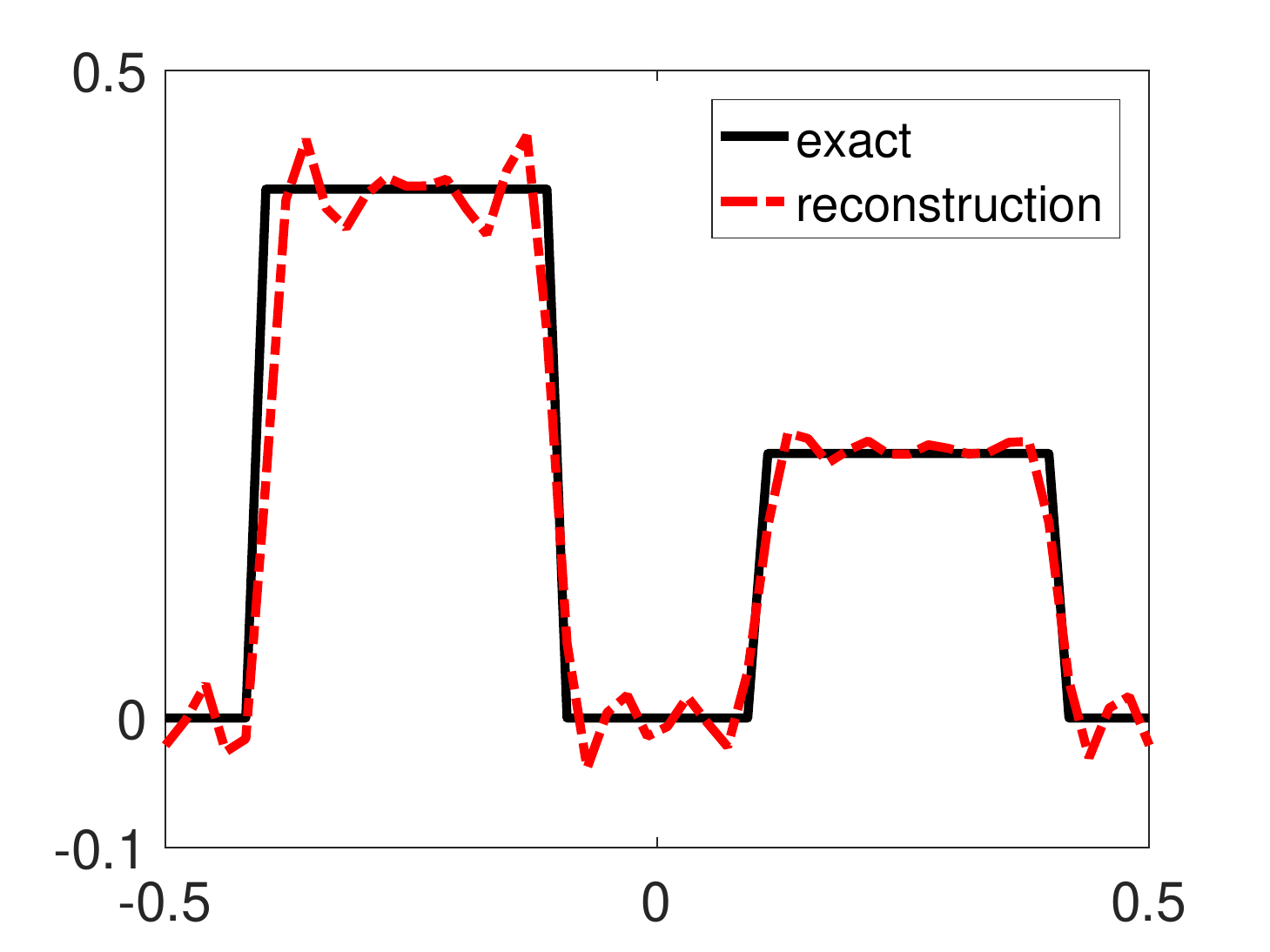}} \hfill
\hfill\subfigure[]{\includegraphics[width=0.32\textwidth]
                   {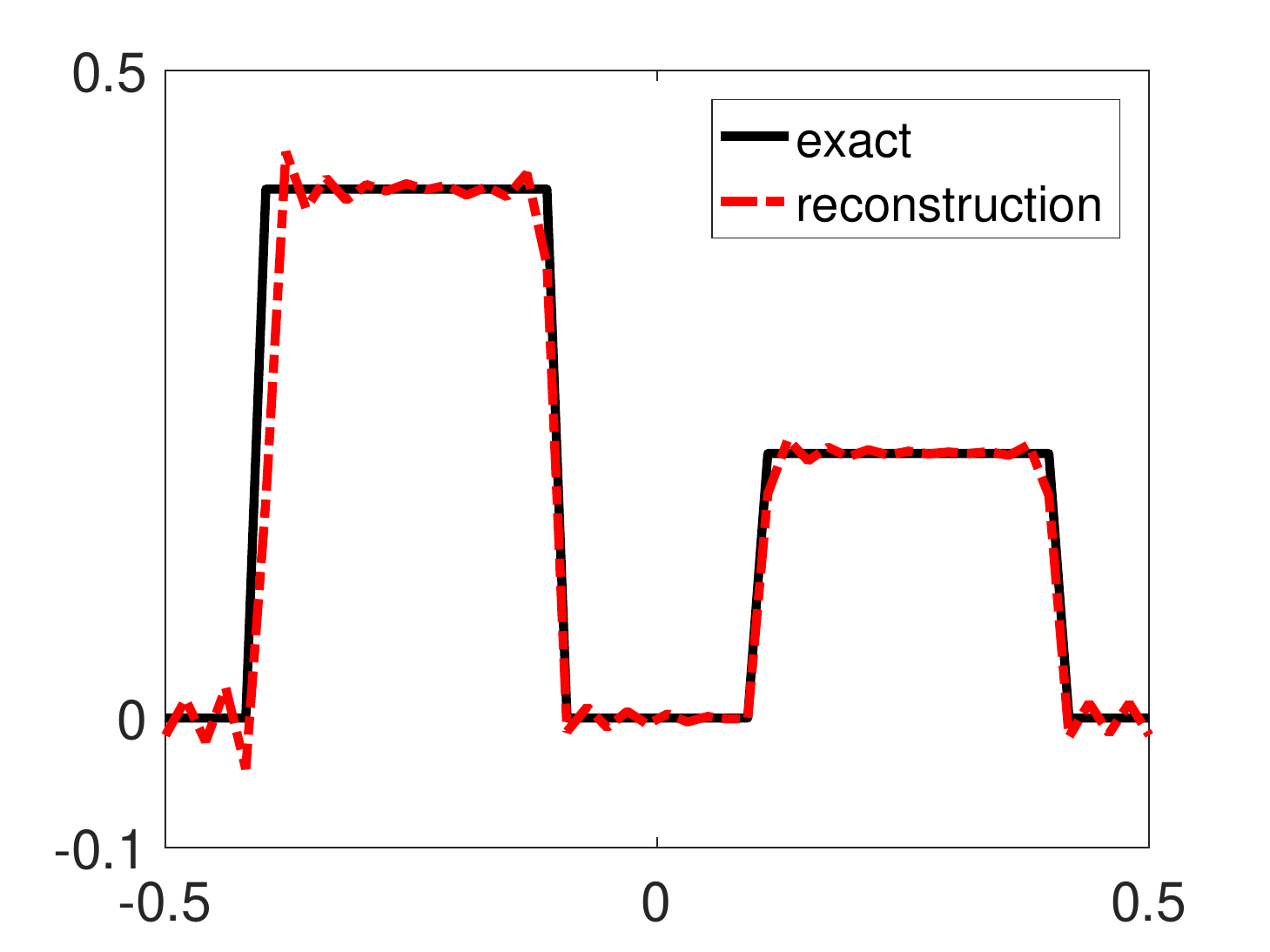}} \hfill\\
\caption{\label{fig:5} Gibbs phenomenon of the reconstructed source $J_1^N$ for different $N$ with $x_2=x_3=0$. (a) $N=5$, (b) $N=15$, (c) $N=25$.}
\end{figure}

Figure \ref{fig:4} shows the contour plots of the exact source and the reconstructions with different truncation order, $N=5, 10, 15, 20, 25$. It is clear that the resolution of the reconstructed results increase as the truncation order $N$ increases. Figure \ref{fig:5} shows the Gibbs phenomenon of the reconstructions over the line $x_2=x_3=0$ with the truncation order $N=5, 15 , 25$, respectively.

\section*{Acknowledgment}

The work of M. Song was supported by the NSFC grant under No. 11671113.
The work of Y. Guo was supported by the NSF grants of China under 11601107, 11671111 and 41474102.
The work of H. Liu was supported by the FRG and startup grants from Hong Kong Baptist University, Hong Kong RGC General Research Funds, 12302415 and 12302017.


\begin{thebibliography}{99}

\bibitem{1Mon} R. Albanese and P. Monk, {\it The inverse source problem for Maxwell's equations}, Inverse Problems, {\bf 22} (2006), 1023--1035.

\bibitem{Ammari2002} H. Ammari, G. Bao and J. Fleming, {\it An inverse source problem for Maxwell's equations in magnetoencephalography}, SIAM J. Appl. Math., {\bf 62} (2002), 1369--1382.

\bibitem{Anastasio} M. Anastasio, J. Zhang, D. Modgil and P. La Rivi, {\it Application of inverse source concepts to photoacoustic tomography}, Inverse Problems, {\bf 23} (2007), 21--35.


\bibitem{Arridge1999} S. Arridge, {\it Optical tomography in medical imaging}, Inverse Problems, {\bf 15} (1999), R41--R93.

\bibitem{Bao2017} G. Bao, P. Li and Y. Zhao,  {\it Stability in the inverse source problem for elastic and electromagnetic waves with multi-frequencies}, (2017), arXiv:1703.03890v1.




\bibitem{ClaKli} C. Clason and M. Klibanov, {\it The quasi-reversibility method for thermoacoustic tomography in a heterogeneous medium}, SIAM J. Sci. Comput., {\bf 30} (2007/08), no. 1, 1--23.

\bibitem{Colton2013} {D.~Colton and R.~Kress}, {\it Inverse Acoustic and Electromagnetic Scattering Theory}, 3rd Edition,
Springer-Verlag, Berlin, 2013.

\bibitem{Bleistein1977} N. Bleistein  and J. Cohen, {\it Nonuniqueness in the inverse source problem in acoustics and electromagnetics}, J. Math. Phys., {\bf 18} (1977), 194--201.


\bibitem{Badia2013} A. El Badia1 and T. Nara, {\it Inverse dipole source problem for time-harmonic Maxwell equations: algebraic algorithm and H$\ddot{o}$lder stability}, Inverse Problems, {\bf 29} (2013), 015007.

\bibitem{Badia2002} A. El Badia and T. Ha-Duong, {\it On an inverse source problem for the heat equation. Application to a pollution detection problem}, J. Inverse Ill-posed Probl., {\bf 10} (2002), 585--99.

\bibitem{Eller2009} M. Eller and N. Valdivia, {\it Acoustic source identification using multiple frequency information}, Inverse Problems, {\bf 25} (2009), 115005.

\bibitem{Fokas2004} A. Fokas, Y. Kurylev and V. Marinakis, {\it The unique determination of neuronal currents in the brain via magnetoencephalography}, Inverse Problems, {\bf 20} (2004), 1067--1082.

\bibitem{He1998} S. He and V. Romanov, {\it Identification of dipole sources in a bounded domain for Maxwell's equations}, Wave Motion, {\bf 28} (1998),  25--40.


\bibitem{Isa2} V. Isakov, {\it Inverse Source Problems}, Mathematical Surveys and Monographs, 34. American Mathematical Society, Providence, 1990.

\bibitem{Klibanov2013} M. Klibanov,  {\it Thermoacoustic tomography with an arbitrary elliptic operator}, Inverse Problems, {\bf 29} (2013), no. 2, 025014.

\bibitem{Liu2015} H. Liu and G. Uhlmann, {\it Determining both sound speed and internal source in thermo- and photo-acoustic tomography}, Inverse Problems, {\bf 31} (2015), no. 10, 105005.

\bibitem{Marengo2004} E. Marengo and A. Devaney , {\it Nonradiating sources with connections to the adjoint problem}, Phys. Rev. E., {\bf 70} (2004), 037601.


\bibitem{Potthast2015} {G. Nakamura and R. Potthast}, {\it Inverse Modeling}, IOP Publishing, Bristol, 2015.

%

 \bibitem{Tai94} C. Tai, {\it Dyadic Green functions in electromagnetic theory}, IEEE, New York,  1994, pp. 48--50.

\bibitem{Lindell88} I. V. Lindell, {\it TE/TM decomposition of electromagnetic sources}, IEEE Transactions on Antennas and Propagation, {\bf 36} (1988),  1382--1388.


\bibitem{Valdivia2012} N. Valdivia, {\it Electromagnetic source identification using
multiple frequency information}, Inverse Problems, {\bf 28} (2012), 115002.


 \bibitem{Wang17} G. Wang, F. Ma, Y. Guo, J. Li,  {\it Solving the multi-frequency electromagnetic inverse source problem by the Fourier method}, (2017), arXiv:1708.00673.

 \bibitem{WangGuo17} X. Wang, Y. Guo, D. Zhang, H. Liu,  {\it Fourier method for recovering acoustic sources from multi-frequency far-field data}, Inverse Problems, {\bf 33} (2017), 035001.




\end{thebibliography}
\end{document}